\documentclass[a4paper,12pt]{article}
\usepackage{amsmath}
\usepackage{amsfonts}
\usepackage{amssymb}
\usepackage{amsthm}
\usepackage{amscd}
\usepackage{tikz}
\usetikzlibrary{matrix,arrows,decorations.pathmorphing}
\usepackage{tikz-cd}
\usepackage{bbm}
\usepackage{enumitem}
\usepackage{graphicx}
\usepackage{setspace}

\setlist[enumerate]{noitemsep}
\setlist[itemize]{noitemsep}

\DeclareMathOperator{\res}{\textup{res}}

\DeclareMathOperator{\Spec}{\textup{Spec}}
\DeclareMathOperator{\Gal}{\textup{Gal}}
\DeclareMathOperator{\Pic}{\textup{Pic}}
\DeclareMathOperator{\Aut}{\textup{Aut}}
\DeclareMathOperator{\tr}{\textup{tr}}
\DeclareMathOperator{\cl}{\textup{cl}}
\DeclareMathOperator{\ab}{\textup{ab}}
\DeclareMathOperator{\et}{\textup{\'et}}

\newcommand{\m}{\mathfrak{m}}
\DeclareMathOperator{\X}{\mathcal{X}}
\DeclareMathOperator{\J}{\mathcal{J}}
\DeclareMathOperator{\Y}{\mathcal{Y}}
\DeclareMathOperator{\Z}{\mathbb{Z}}

\DeclareMathOperator{\Q}{\mathbb{Q}}

\DeclareMathOperator{\B}{\mathbb{B}}
\DeclareMathOperator{\oh}{\mathcal{O}}
\DeclareMathOperator{\Sp}{\textup{Sp}}
\newcommand{\sexact}[5]{\begin{CD}1 @>>> #1 @>#2>> #3 @>#4>> #5 @>>> 1\end{CD}}
\newcommand{\sexactab}[5]{\begin{CD}0 @>>> #1 @>#2>> #3 @>#4>> #5 @>>> 0\end{CD}}
\newcommand{\ob}[1]{\mkern 1.5mu\overline{\mkern-1.5mu#1\mkern-1.5mu}\mkern 1.5mu}

\newtheorem{theorem}{Theorem}[section]
\newtheorem{maintheorem}{Theorem}

\newtheorem{proposition}[theorem]{Proposition}
\newtheorem{lemma}[theorem]{Lemma}
\newtheorem{corollary}[theorem]{Corollary}
\newtheorem*{conjecture*}{Conjecture}
\theoremstyle{definition}
\newtheorem{definition}[theorem]{Definition}
\newtheorem{remark}[theorem]{Remark}

\renewenvironment{abstract}{
	\noindent\textsc{\abstractname.}}

\usepackage{tocloft}

\usepackage{titlesec}
\titleformat{\section}[block]{\centering\normalsize\sc}{\thesection}{1em}{}
\titleformat{\subsection}[runin]{\normalsize\bf}{\thesubsection}{1em}{}[.]

\usepackage[numbers]{natbib}
\usepackage[OT2,T1]{fontenc}
\DeclareSymbolFont{cyrletters}{OT2}{wncyr}{m}{n}
\DeclareMathSymbol{\Sha}{\mathalpha}{cyrletters}{"58}

\title{\bf\normalsize\MakeUppercase{On the Birational Section Conjecture over Finitely Generated Fields}}
\author{\normalsize\sc Mohamed Sa\"idi and Michael Tyler }
\date{}

\makeindex

\begin{document}

\maketitle

\begin{abstract}
\small We investigate the birational section conjecture for curves over function fields of characteristic zero and prove that the conjecture holds over finitely generated fields over $\Q$ if it holds over number fields.
\end{abstract}

\tableofcontents

\section{Introduction and statement of results}

\subsection{The birational section conjecture}

For a smooth, geometrically connected, projective curve $X$ over a characteristic zero field $k$, we define the \emph{absolute Galois group of $X$} to be the group
\[G_X:=\Gal(\overline {k(X)}/k(X))\]
where $k(X)$ denotes the function field of $X$ and $\overline {k(X)}$ is an algebaic closure of $k(X)$.  The Galois group $G_X$ fits into an exact sequence
\[\sexact{G_{X_{\bar k}}}{}{G_X}{}{G_k}\]
where $G_k := \Gal(\bar k|k)$, $\bar k$ being the algebraic closure of $k$ in $\overline {k(X)}$, and $G_{X_{\bar k}}:=\Gal (\overline {K(X)}/K(X) \cdot \bar k)$.  Let $x \in X(k)$ be a $k$-rational point, and let $\tilde x$ be a valuation of $\overline {k(X)}$ extending the valuation $\nu_x$ of $k(X)$ corresponding to $x$.  We will refer to $\tilde x$ as an \emph{extension of $x$ to $\overline {k(X)}$}.  The \emph{decomposition group} $D_{\tilde x}$ of $\tilde x$ fits into a commutative diagram of exact sequences
\[\begin{tikzcd}
1 \arrow{r}{} &I_{\tilde x} \arrow{r}{}\arrow[hookrightarrow]{d}{} &D_{\tilde x} \arrow{r}{}\arrow[hookrightarrow]{d}{} &G_{k(x)} \arrow{r}{}\arrow[equals]{d}{} &1\\
1 \arrow{r}{} &G_{X_{\bar k}} \arrow{r}{} &G_X \arrow{r}{} &G_k \arrow{r}{} &1
\end{tikzcd}\]
where $I_{\tilde x}$ is the \emph{inertia group} at $\tilde x$.  We will refer to a splitting of the lower sequence in the above diagram as a \emph{section of $G_X$}.  A splitting of the upper exact sequence naturally defines a section $s_{\tilde x} : G_k \to G_X$ of $G_X$, with image contained in $D_{\tilde x}$.  
\begin{definition}\label{geometricgaloissections}
We say a section $s$ of $G_X$ is \emph{geometric} if its image $s(G_k)$ is contained in a decomposition group $D_{\tilde x}$ for some $k$-rational point $x\in X(k)$ and some extension $\tilde x$ of $x$ to $\overline {k(X)}$.  In this case, we say that the section $s$ \emph{arises from the point} $x$.
\end{definition}

The birational analogue of Grothendieck's anabelian section conjecture for \'etale fundamental groups may be stated as follows.

\begin{conjecture*} 
Let $k$ be a finitely generated field over $\Q$, and let $X$ be a smooth, projective, geometrically connected curve over $k$.  Then every section of $G_X$ is geometric and arises from a unique $k$-rational point $x\in X(k)$.
\end{conjecture*}

We will refer to this as the \emph{birational section conjecture} or \textbf{BSC}.  One may consider the statement for more general fields $k$, so we establish the following terminology.

\begin{definition}\label{bscholds}
\begin{enumerate}
\item Let $X$ be a smooth, geometrically connected, projective curve over a field $k$.  We say the birational section conjecture (or \textbf{BSC}) \emph{holds for $X$} if every section of $G_X$ is geometric and arises from a unique $k$-rational point $x\in X(k)$.
\item For a field $k$, we say the birational section conjecture (or \textbf{BSC}) \emph{holds over $k$} if the \textbf{BSC} holds for every smooth, geometrically connected, projective curve over $k$.
\end{enumerate}
\end{definition}

\begin{remark}\label{bscuniqueness}
To prove that the \textbf{BSC} holds for $X$ it suffices to prove that every section of $G_X$ arises from a $k$-rational point $x \in X(k)$.  This is necessarily the unique such point, since decomposition subgroups of $G_X$ associated to distinct rank $1$ valuations of $\overline {k(X)}$ intersect trivially \cite[Corollary 12.1.3]{CONF} (in loc. cit. $k$ is a global field but the same argument of proof works over any field). 
\end{remark}

It is hoped that the birational section conjecture might be used to prove the Grothendieck anabelian section conjecture for \'etale fundamental groups, via the theory of ``cuspidalisation'' of sections of arithmetic fundamental groups \cite{saidicusp}.  Conversely, the anabelian section conjecture for $\pi_1$ implies the birational section conjecture, as follows easily from the ``limit argument'' of Akio Tamagawa \cite[Proposition 2.8 (iv)]{tamagawa}.  At present the \textbf{BSC} over finitely generated fields over $\Q$, as well as the anabelian section conjecture for $\pi_1$, are quite open. Few examples are known of curves over number fields for which the {\bf BSC} holds. More precisely, with the notations and assumptions in Definition 1.2(i), if $k$ is a number field, $X$ is hyperbolic, $J(k)$ is finite where $J$ is the jacobian of $X$, and the Shafarevich-Tate group of $J$ is finite, then the \textbf{BSC} holds for $X$ (see \cite[Remark 8.9]{stoll} for related examples). Also, some conditional results on the birational section conjecture over number fields of small degree are known (see \cite{hoshi}). Note that a $p$-adic analog of the birational section conjecture holds by \cite{koenigsmann}. To the best of our knowledge no result on the birational section conjecture is known for curves over finitely generated fields over $\Q$ of positive transcendence degree.

\subsection{Statement of the Main Theorems}

In this paper we investigate the birational section conjecture over function fields.  We prove that, for a certain class of fields $k$ of characteristic zero, and under the condition of finiteness of certain Shafarevich-Tate groups, proving that the \textbf{BSC} holds over function fields over $k$ reduces to proving that it holds over finite extensions of $k$.  This class of fields contains, in particular, the finitely generated extensions of $\Q$, and for such fields we show further that the statement is independent of finiteness of the Shafarevich-Tate groups.  The result therefore reduces the \textbf{BSC} over finitely generated extensions of $\Q$ unconditionally to the case of number fields.

Our approach stems from the proof in \cite{saidiSCOFF} of a similar result for the section conjecture for \'etale fundamental groups.  Let us start by describing the aforementioned class of fields, which was introduced in \cite[Definition 0.2]{saidiSCOFF}.

\begin{definition}\label{conditions}
For a field $k'$ of characteristic zero, consider the following conditions on $k'$.
\begin{enumerate}
\item The {\bf BSC} holds over $k'$.
\item For every prime integer $\ell$, the $\ell$-cyclotomic character $\chi_{\ell}:G_{k'}\rightarrow\Z^{\times}_{\ell}$ is non-Tate, meaning that any $G_{k'}$-map $\Z_{\ell}(1)\rightarrow T_{\ell} A$, for some abelian variety $A$ over $k'$ and its $\ell$-adic Tate module $T_{\ell}A$, vanishes.
\item Given an abelian variety $A$ over $k'$, any quotient $A(k')\twoheadrightarrow D$ of the group of $k'$-rational points $A(k')$ satisfies the following:
\begin{enumerate}
\item The natural map $D\rightarrow \widehat D$ is injective, where $\widehat D:=\varprojlim_{N \geq 1} D/ND$.
\item The $N$-torsion subgroup $D[N]$ of $D$ is finite for all $N\geq 1$, and the Tate module $TD$ is trivial (cf. Notation).
\end{enumerate}
\item Given a separated, smooth, connected curve $C$ over $k'$ with function field $K=k'(C)$, $K$ admits the structure of a Hausdorff topological field, so that $X(K)$ is compact for any smooth, geometrically connected, projective, hyperbolic curve $X$ over $K$.
\item Given a separated, smooth, and connected (not necessarily projective) curve $C$ over $k'$ with function field $K=k'(C)$ and a finite morphism $\tilde C\rightarrow C$ with $\tilde C$ separated and smooth, then the following holds.  If $\tilde C_c(k'(c))\ne\emptyset$ for all closed points $c\in C^{\textup{cl}}$, where $k'(c)$ denotes the residue field at $c$ and $\tilde C_c$ is the scheme-theoretic inverse image of $c$ in $\tilde C$, then $\tilde C(K)\ne\emptyset$.
\end{enumerate}
For a field $k$ of characteristic zero, we say that $k$ \emph{strongly satisfies} one of the above conditions (i), (ii), (iii), (iv) and (v) if this condition is satisfied by any finite extension $k'|k$ of $k$.
\end{definition}

Condition (i) is expected to hold for all finitely generated fields over $\Bbb Q$ by the \textbf{BSC}.
Conditions (ii)-(v) are satisfied by finitely generated fields over $\Q$ (cf. \cite{saidiSCOFF}, discussion after Definition 0.2).
More precisely, (ii) follows from the theory of weights, (iii) follows from the Mordell-Weil and Lang-N\'eron Theorems, and (iv) follows (for the discrete topology) from Mordell's conjecture: Faltings' Theorem and N\'eron's specialisation Theorem. Condition (v) is satisfied by Hilbertian fields (cf. \cite[Lemma 4.1.5]{saidiSCOFF}), in particular (v) holds for finitely generated fields.

\begin{definition}\label{sha}
Let $k$ be a field of characteristic zero and $C$ a smooth, separated, connected curve over $k$ with function field $K$.  Let $\mathcal{A}\rightarrow C$ be an abelian scheme with generic fibre $A := \mathcal{A} \times_C K$.  For each closed point $c\in C^{\textup{cl}}$ denote by $K_c$ the completion of $K$ at $c$, and write $A_c := A \times_K K_c$.  
We define the Shafarevich-Tate group
\[\Sha(\mathcal{A}) := \ker(H^1(G_K,A)\rightarrow\prod_{c\in C^{\textup{cl}}}H^1(G_{K_c},A_c))\]
where the product is taken over all the closed points of $C$.
\end{definition}

We now state our first two main Theorems.

\begin{maintheorem}
Let $k$ be a field of characteristic zero that satisfies conditions (iv) and (v) of Definition \ref{conditions},
and strongly satisfies conditions (i), (ii), and (iii) of Definition \ref{conditions}.  Let $C$ be a smooth, separated, connected curve over $k$ with function field $K$.  Let $\X\rightarrow C$ be a flat, proper, smooth relative curve, with generic fibre $X:=\X\times_C K$ which is a geometrically connected hyperbolic curve over $K$ such that $X(K)\ne\emptyset$.  Denoting by $\mathcal{J}:=\textup{Pic}^0_{\X/C}$ the relative Jacobian of $\X$, assume $T\Sha(\mathcal{J})=0$.  Then the {\bf BSC} holds for $X$.
\end{maintheorem}

\begin{maintheorem}
Let $k$, $C$ and $K$ be as in Theorem A, and assume further that $k$ strongly satisfies conditions (iv) and (v) of Definition \ref{conditions}.  For any finite extension $L$ of $K$, let $C^L$ denote the normalisation of $C$ in $L$.  Assume that for any such finite extension $L$ and any flat, proper, smooth relative curve $\Y \to C^L$ we have $T\Sha(\J_{\!\Y}) = 0$, where $\J_{\!\Y} := \Pic^0_{\Y/C^L}$ is the relative Jacobian of $\Y$.  Then the {\bf BSC} holds over all finite extensions of $K$.
\end{maintheorem}

In the context of $\Q$, this means that if the birational section conjecture holds over all fields of some fixed transcendence degree over $\Q$ then it holds over all fields of transcendence degree one higher, provided that finiteness of $\Sha$ holds.  By induction this means that, under the assumption on the relevant Shafarevich-Tate groups, if the birational section conjecture holds over number fields then it holds over all finitely generated fields over $\Q$. As a consequence of one of the main results in \cite{saiditamagawa}, asserting the finiteness of $\Sha$ for isotrivial abelian varieties over finitely generated fields, and using Theorem A, we can prove the following.

\begin{maintheorem} Assume that the {\bf BSC} holds over all number fields. Then the {\bf BSC} holds over all finitely generated fields over $\Bbb Q$.
\end{maintheorem}

Thus Theorem C reduces the proof of the birational section conjecture to the case of number fields.

\subsection{Guide to the proof of the Main Theorems}

Our approach to proving the above Theorems is inspired by the method in \cite{saidiSCOFF}, and relies on a local-global argument.  This requires studying the properties of sections of absolute Galois groups of curves over local fields of equal characteristic zero and over function fields of curves in characteristic zero.  We relate these two settings by investigating ``\'etale abelian sections'', from where arises the constraint on the Shafarevich-Tate groups.  

{\it To the proof of Theorems A and B}. With the notations and assumptions in Theorem A, let $s:G_K \to G_X$ be a section. For each closed point $c\in C^{\cl}$, using the fact that $k$ strongly satisfies condition (ii), as well as a specialisation theorem for absolute Galois groups which may be of interest independently of the context of this paper (cf. \S 3.1), we prove that $s$ naturally induces a birational section $\bar s_c:G_{k(c)}\to G_{\X_c}$ of the absolute Galois group $G_{\X_c}$ of the fibre $\X_c$ at $c$.  The fact that $k$ strongly satisfies condition (i) implies that the section $\bar s_c$ is geometric and arises from a unique rational point $x_c\in \X_c(k(c))$ (the unicity follows from the fact that $k$ strongly satisfies condition (iii)(a)).

Next, we consider the \'etale abelian section $s^{\ab}:G_K\to \pi_1(X)^{(\ab)}$ induced by $s$, where $\pi_1(X)^{(\ab)}$ is the geometrically abelian fundamental group of $X$, and similarly the \'etale abelian sections $\bar s_c^{\ab}:G_{k(c)}\to \pi_1(\X_c)^{(\ab)}$ induced by $\bar s_c$, $\forall c\in C^{\cl}$ (cf. \S 2.2). Thus, $s^{\ab}$ and $\bar s_c^{\ab}$ can be viewed as elements of $H^1(G_K,TJ)$, and $H^1(G_{k(c)},T\J_c)$, where $J$ and $\J_c$ are the jacobians of $X$ and $\X_c$ respectively, and $T$ indicates their Tate modules (cf. loc. cit.).  We have the following commutative diagram of Kummer exact sequences (cf. $\S2.2$)
\[\begin{tikzpicture}
\matrix(m)[matrix of math nodes,
row sep=3em, column sep=2em,
text height=2ex, text depth=0.25ex]
{0 & \widehat{J(K)} & H^1(G_K,TJ) & TH^1(G_K,J) & 0\\
0 & \displaystyle\prod_{c\in C^{\cl}}\widehat{J_c(K_c)} & \displaystyle\prod_{c\in C^{\cl}}H^1(G_{K_c},J_c) & \displaystyle\prod_{c\in C^{\cl}}TH^1(G_{K_c},J_c) & 0\\};
\path[->]
(m-1-1) edge (m-1-2);
\path[->]
(m-1-2) edge (m-1-3) edge (m-2-2);
\path[->]
(m-1-3) edge (m-1-4) edge (m-2-3);
\path[font=\scriptsize,->]
(m-1-4) edge (m-1-5) edge node[left]{} (m-2-4);
\path[->]
(m-2-1) edge (m-2-2);
\path[->]
(m-2-2) edge (m-2-3);
\path[->]
(m-2-3) edge (m-2-4);
\path[->]
(m-2-4) edge (m-2-5);
\end{tikzpicture}\]
where $J_c:=J\times _KK_c$, and $K_c$ is the completion of $K$ at $c$. Moreover, $(s_c^{\ab})_{c\in C^{\cl}}=(x_c)_{c\in C} \in \prod_{c\in C^{\cl}}\widehat{J_c(K_c)}$ via the injective maps $\prod_{c\in C^{\cl}}\X_c(k(c))\to \prod_{c\in C^{\cl}}\widehat{\J_c(k(c))}\to \prod_{c\in C^{\cl}}\widehat{J_c(K_c)}$, where the last map is induced by the inflation map, and the first map is injective as $k$ strongly satisfies condition (iii)(a).  The kernel of the right vertical map $TH^1(G_K,J)\to \prod_{c\in C^{\cl}}TH^1(G_{K_c},J_c)$ is the Tate module of the Shafarevich-Tate group $\Sha(\mathcal{J})$, which is trivial by our assumption that $\Sha(\mathcal{J})$ is finite. Thus $s^{\ab}\in  \widehat{J(K)}$ by the above diagram. We then prove the following 
(cf. Lemma 5.1).

\begin{maintheorem} We use the above notations. Assume that $k$ strongly satisfies the conditions (i), (ii), and (iii) of Definition 1.4.  
Then the homomorphism $J(K)\rightarrow\widehat{J(K)}$ is injective and $s^{\ab}$ is contained in $J(K)$.
\end{maintheorem}

Furthermore, the natural map $\prod_{c\in C^{\cl}}\X_c(k(c))\to \prod_{c\in C^{\cl}}\J_c(k(c))$ is injective (this follows from condition (iii)(a)), and inside
$\prod_{c\in C^{\cl}}\J_c(k(c))$ the equality 
$$\prod_{c\in C^{\cl}}\X_c(k(c))\cap J(K)=X(K)$$ 
holds (the map $J(K)\to \prod_{c\in C^{\cl}}\J_c(k(c))$ is injective. See Lemma 5.2 and its proof). 
Thus $s^{\ab}=x\in X(K)$ arises from a unique rational point $x$ since $s_c=x_c\in \X_c(k(c))$ for all $c\in C^{\cl}$.  To finish the proof of Theorem A, using a limit argument due to Tamagawa relying on the fact that $k$ satisfies condition (iv), it suffices to show that every neighbourhood of the section $s$ has a $K$-rational point. This is achieved using the fact that $k$ satisfies condition (v).

The proof of Theorem B follows easily from Theorem A, and standard facts on the behaviour of birational Galois sections with respect to finite base change.

{\it To the proof of Theorem C}. We argue by induction on the transcendence degree and assume that the \textbf{BSC} holds over all finitely generated fields over $\Bbb Q$ with transcendence degree $<n$, where $n\ge 1$ is an integer. To prove that the \textbf{BSC} holds for curves over a finitely generated field $K$ of transcendence degree $n$ we first reduce to the case of the projective line $\Bbb P^1_K$. Given a section $s:G_K\to G_{\Bbb P^1_K}$ we prove that $s$ has a neighbourhood $Y\to \Bbb P^1_K$ with $Y$ hyperbolic and isotrivial. For such a curve it is proven in \cite{saiditamagawa} that $\Sha (\J_Y)$ is finite, where $\J_Y\to C$ is the relative jacobian of a suitable model $\Y\to C$ of $Y$ as in the statement of Theorem A. We can then (after passing to an appropriate finite extension of $K$) apply Theorem A to conclude that the section $s_Y:G_K\to G_Y$ induced by $s$, and a fortiori the section $s$, is geometric.

\subsection{Guide to the paper}

The layout of the paper is as follows.  In \S 2 we will recall some properties of \'etale and geometrically abelian fundamental groups and their sections, which will be necessary for the proofs of our main Theorems.

In \S 3 we work in a local setting.  We consider a flat, proper, smooth relative curve over the spectrum of a complete discrete valuation ring with residue characteristic zero, and explain how to define a specialisation homomorphism of absolute Galois groups (Theorem \ref{galspec}) using the specialisation homomorphism of fundamental groups of affine curves.  
We apply this in \S \ref{ramificationofsections} to study the specialisation of sections, and the phenomenon of ramification of sections.

In \S 4 we return to the global setting of Theorems A and B, and consider curves over function fields of characteristic zero.  In \S \ref{specialisationofsections} we explain how to pass to the local setting and apply the results from \S 3.  In \S \ref{etaleabeliansections} we consider \'etale abelian sections and show how to apply a local-global principle, during which we encounter the issue of finiteness of the Shafarevich-Tate group (Proposition \ref{sabjkhat}).

In \S 5 we use the results of \S 3 and \S 4 to prove Theorems A and B, and apply these, along with one of the main results of \cite{saiditamagawa}, to prove Theorem C.

\subsection*{Notation}

\begin{itemize}
\item For a scheme $X$, we will denote the set of closed points of $X$ by $X^{\cl}$. 
\item Given a ring $R$ and morphisms of schemes $X \to Y$ and $\Spec R \to Y$, we will denote the fibre product $X \times_Y \Spec R$ by $X_R$.
\item For a field $k$ and a given algebraic closure $\bar k$ of $k$, we will denote the absolute Galois group $\Gal(\bar k|k)$ by $G_k$.
\item For an abelian group $A$ and a positive integer $N$, we denote by $A[N]$ the kernel of the homomorphism $N : A \to A$, $a \mapsto N \cdot a$.
\item Given an abelian group $A$, we denote by $\widehat{A}$ the inverse limit $\widehat{A} := \varprojlim_{N \geq 1} A/NA$ and by $TA := \varprojlim_{N \geq 1} A[N]$ the Tate module of $A$.
\item Given an \emph{abelian variety} $B$ over a field $k$, an algebraic closure $\bar k$ of $k$, we will denote by $TB$ the Tate module of the abelian group $B(\bar k)$.
\end{itemize}

\section{Sections of \'etale and geometrically abelian fundamental groups}

\subsection{Sections of \'etale fundamental groups}

Let $k$ be a field of characteristic zero, $X$ a smooth, geometrically connected, projective curve over $k$, and $U\subset X$ an open subset with complement $S := X \setminus U$ (when considering an open $U\subset X$ we assume $U$ non-empty).  
Write $K=k(X)$ for the function field of $X$. Let $z$ be a geometric point of $U$ with image in the generic point. Thus $z$ determines algebraic closures $\ob{K}$, and $\bar k$, of $K$ and $k$ respectively, as well as a geometric point of $U_{\bar k}$ that we denote $\bar z$. The \emph{\'etale fundamental group} $\pi_1(U, z)$ of $U$ fits into an exact sequence:
\[\sexact{\pi_1(U_{\bar k},\bar z)}{}{\pi_1(U,z)}{}{G_k=\Gal(\bar k/k)}\]
We will refer to this as the \emph{fundamental exact sequence} of $U$.  The fundamental group $\pi_1(U_{\bar k},\bar z)$ will be called the \emph{geometric fundamental group} of $U$.

\begin{definition}\label{universalcover}
A \emph{universal pro-\'etale cover} $\tilde U\rightarrow U$ of $U$ is an inverse system of finite \'etale covers $\{V_i\rightarrow U\}_{i\in I}$, corresponding to open subgroups of
$\pi_1(U,z)$, such that for any \'etale cover $V\rightarrow U$, corresponding to an open subgroup of $\pi_1(U,z)$, there is a $U$-morphism $V_i\rightarrow V$ for some $i\in I$.  A (closed) \emph{point} $\tilde x$ of the universal pro-\'etale cover $\tilde U$ is a compatible system of (closed) points $\{x_i\in V_i\}_{i\in I}$.
\end{definition}

Fix a universal pro-\'etale cover $\tilde U \to U$, and let $\{V_i \to U\}_{i\in I}$ be the inverse system of \'etale covers defining it.

\begin{definition}\label{normalisation}
Let $Y_i \to X$ denote the unique extension of $V_i \to U$ to a finite morphism of smooth, connected, projective curves over $k$.  The inverse system of morphisms $\{Y_i \to X\}_{i\in I}$ will be called the \emph{normalisation of $X$ in $\tilde U$}, denoted $\tilde X_{\tilde U} \to X$.  A (closed) \emph{point} $\tilde x$ of $\tilde X_{\tilde U}$ is a compatible system of (closed) points $\{x_i\in Y_i\}_{i\in I}$.
\end{definition}

\begin{definition}\label{decompinertiapi1}
With the above notation, let $x\in X^{\cl}$ be a closed point and $\tilde x$ a point above $x$ in the normalisation $\tilde X_{\tilde U}$ of $X$ in $\tilde U$.  The \emph{decomposition group} $D_{\tilde x}$ of $\tilde x$ is the stabiliser of $\tilde x$ under the action of $\pi_1(U,z)$ (which acts naturally on the points of $\tilde X_{\tilde U}$).  
The \emph{inertia group} $I_{\tilde x}$ is the kernel of the natural projection $D_{\tilde x}\twoheadrightarrow G_{k(x)}= \Aut(\bar k/k(x))$, where $k(x)$ is the residue field at $x$.
\end{definition}

It follows immediately from the definition that decomposition and inertia groups for different choices of $\tilde x$ above $x$ are conjugate, that is, for any $\sigma\in\pi_1(U,z)$ we have $I_{\sigma\tilde x}=\sigma I_{\tilde x}\sigma^{-1}$ and $D_{\sigma\tilde x}=\sigma D_{\tilde x}\sigma^{-1}$.  The inertia group $I_{\tilde x}$ is trivial if $x \in U$, while if $x \in S$ it is isomorphic as a $G_{k(x)}$-module to the Tate twist $\hat\Z(1)$.

Recall that the curve $U$ is \emph{hyperbolic} if the geometric fundamental group $\pi_1(U_{\bar k}, \bar z)$ is non-abelian.  Denoting by $g$ the genus of $X$, this occurs exactly when $2 - 2g - \deg_k S < 0$, where $S$ is regarded as a reduced effective divisor on $X$.  We refer to \cite[Lemma 1.5]{hoshimochizuki} for a proof of the following statement.

\begin{lemma}\label{inertiaproperties}
With the above notation, assume further that $U$ is hyperbolic.  Then any two inertia subgroups $I_{\tilde x}, I_{\tilde x'} \subset \pi_1(U, z)$ corresponding to distinct points $\tilde x\ne\tilde x'$ of $\tilde X_{\tilde U}$ intersect trivially.
\end{lemma}

Let $x \in X(k)$ be a $k$-rational point of $X$, and let $\tilde x$ be a point above $x$ in the normalisation $\tilde X_{\tilde U}$ of $X$ in $\tilde U$.  There is a commutative diagram
of exact sequences
\begin{equation}\label{decompinjectpi1}
\begin{tikzcd}
1 \arrow{r}{} &I_{\tilde x} \arrow{r}{}\arrow[hookrightarrow]{d}{} &D_{\tilde x} \arrow{r}{}\arrow[hookrightarrow]{d}{} &G_{k(x)} \arrow{r}{}\arrow[equals]{d}{} &1\\
1 \arrow{r}{} &\pi_1(U_{\bar k},\bar z) \arrow{r}{} &\pi_1(U,z) \arrow{r}{} &G_{k} \arrow{r}{} &1
\end{tikzcd}
\end{equation}
where the middle vertical map is the natural inclusion, via which a section of the upper exact sequence naturally defines one of the (lower) fundamental exact sequence.  When $x \in U$ the inertia group $I_{\tilde x}$ is trivial, hence there is an isomorphism $D_{\tilde x} \simeq G_{k(x)}=G_k$ (induced by the projection $\pi_1(U,z)\twoheadrightarrow G_k$) and thus $D_{\tilde x}$ gives rise naturally to a section of $\pi_1(U, z)$ in this case.  Action by $\pi_1(U_{\bar k}, \bar z)$ permutes the points $\tilde x$ of $\tilde U$ above $x$, so $x$ in fact induces a conjugacy class of sections of the fundamental exact sequence.

It is well known that the upper exact sequence in diagram (\ref{decompinjectpi1}) also splits when $x$ is contained in the complement $S = X \setminus U$.

\begin{definition}\label{cuspidal}
A section $s:G_k\rightarrow\pi_1(U,z)$ is called \emph{cuspidal} if it factors through $D_{\tilde x}$ for some (necessarily $k$-rational) point $x \in S$ and some $\tilde x\in\tilde X_{\tilde U}$ above $x$.
\end{definition}

Now let $\xi : \Spec\Omega \to X$ be a geometric point with image the generic point of $X$.  The geometric point $\xi$ naturally determines a choice of an algebraic closure $\overline {K}$ 
of $K$ and, for each open subset $U \subset X$, a geometric point $\xi:\Spec\Omega\to U$ with image the generic point of $U$.  The following is well-known. 

\begin{lemma}\label{Gallim}
With $X$, $K$ and $\xi$ as above, there is a canonical isomorphism
\[G_X:=\Gal (\overline {K}/K)\simeq\varprojlim_{U\subset X\textup{ open}}\pi_1(U,\xi)\]
where the limit is taken over all the open subsets of $X$, partially ordered by inclusion.
\end{lemma}

In particular, for any open subset $U \subset X$, the fundamental group $\pi_1(U, \xi)$ is naturally a quotient of $G_X$.  In fact we have a commutative diagram
\[\begin{tikzcd}
1 \arrow{r}{} &G_{X_{\bar k}} \arrow{r}{}\arrow[twoheadrightarrow]{d}{} &G_X \arrow{r}{}\arrow[twoheadrightarrow]{d}{} &G_k \arrow{r}{}\arrow[equals]{d}{} &1\\
1 \arrow{r}{} &\pi_1(U_{\bar k},\bar\xi) \arrow{r}{} &\pi_1(U,\xi) \arrow{r}{} &G_k \arrow{r}{} &1
\end{tikzcd}\]
where the middle vertical map is the natural projection, via which a section $s : G_k \to G_X$ of $G_X$ naturally induces a section $s_U : G_k \to \pi_1(U, \xi)$
of the projection $\pi_1(U, \xi)\twoheadrightarrow G_k$.  
Thus, by Lemma \ref{Gallim}, a section $s$ of $G_X$ determines, and is determined by, a compatible system of sections $s_U : G_k \to \pi_1(U, \xi)$ for the open subsets $U \subset X$.

\subsection{Geometrically abelian fundamental groups}

Let $k$ be a field of characteristic zero, and let $X$ be a smooth, geometrically connected, projective curve over $k$ such that $X(k) \ne \emptyset$.  Let $z:\Spec\Omega\rightarrow X$ be a geometric point with value in the generic point, which determines an algebraic closure $\bar k$ of $k$ and a geometric point $\bar z$ of $X_{\bar k}$.  Denote by $\pi_1(X_{\bar k}, \bar z)^{\ab}$ the maximal abelian quotient of the geometric fundamental group of $X$.  The \emph{geometrically abelian quotient} of $\pi_1(X, z)$, denoted $\pi_1(X, z)^{(\ab)}$, is defined by the pushout diagram
\[\begin{tikzcd}
1 \arrow{r}{} &\pi_1(X_{\bar k},\bar z) \arrow{r}{}\arrow[twoheadrightarrow]{d}{} &\pi_1(X,z) \arrow{r}{}\arrow[twoheadrightarrow]{d}{} &G_k \arrow{r}{}\arrow[equals]{d}{} &1\\
1 \arrow{r}{} &\pi_1(X_{\bar k},\bar z)^{\ab} \arrow{r}{} &\pi_1(X,z)^{(\ab)} \arrow{r}{} &G_k \arrow{r}{} &1
\end{tikzcd}\]
where the upper row is the fundamental exact sequence.  By commutativity of this diagram, a section $s : G_k \to \pi_1(X,z)$ of the \'etale fundamental group of $X$ naturally induces a section $s^{\ab} : G_k \to \pi_1(X, z)^{(\ab)}$ of the geometrically abelian quotient, which we will call the \emph{\'etale abelian section} induced by $s$.

Fix a $k$-rational point $x_0\in X(k)$.  Let $J$ denote the Jacobian of $X$, and let $\iota:X\rightarrow J$ be the closed immersion mapping $x_0$ to the zero section of $J$.  Note that $\iota$ maps an arbitrary $k$-rational point $x \in X(k)$ to the class of the degree zero divisor $[x]-[x_0]$.  Moreover, it induces a commutative diagram of exact sequences
\begin{equation}\label{pi1xpi1j}
\begin{tikzpicture}[descr/.style={fill=white}, baseline=(current  bounding  box.center)]
\matrix(m)[matrix of math nodes,
row sep=3em, column sep=2.5em,
text height=1.5ex, text depth=0.25ex]
{1 & \pi_1(X_{\bar k},\bar z)^{\ab} & \pi_1(X,z)^{(\ab)} & G_k & 1\\
1 & \pi_1(J_{\bar k},\bar z) & \pi_1(J,z) & G_k & 1\\};
\path[->]
(m-1-1) edge (m-1-2);
\path[->]
(m-1-2) edge (m-1-3);
\path[->,font=\scriptsize]
(m-1-2) edge node[below,rotate=90]{$\sim$} (m-2-2);
\path[->]
(m-1-3) edge (m-1-4);
\path[->,font=\scriptsize]
(m-1-3) edge node[below,rotate=90]{$\sim$} (m-2-3);
\path[->]
(m-1-4) edge (m-1-5);
\path[-]
(m-1-4) edge[double distance=2pt] (m-2-4);
\path[->]
(m-2-1) edge (m-2-2);
\path[->]
(m-2-2) edge (m-2-3);
\path[->]
(m-2-3) edge (m-2-4);
\path[->]
(m-2-4) edge (m-2-5);
\end{tikzpicture}
\end{equation}
where the vertical maps are isomorphisms.  

Hence there is an identification of $G_k$-modules $\pi_1(X_{\bar k},\bar z)^{\ab} \simeq \pi_1(J_{\ob{k}}, \bar z) \simeq TJ$, where $TJ$ is the Tate module of $J$ (see Notation).  We fix a base point of the torsor of splittings of the upper exact sequence in (\ref{pi1xpi1j}) to be the splitting arising from the rational point $x_0$, and the corresponding base point of the torsor of splittings of the lower exact sequence in (\ref{pi1xpi1j}), and identify the set of $\pi_1(X_{\bar k},\bar z)^{\ab}$-conjugacy classes of sections of the upper exact sequence in (\ref{pi1xpi1j}) with the Galois cohomology group $H^1(G_k, TJ)$. By functoriality of the fundamental group, any $k$-rational point $x\in X(k)$ induces a section $s_x:G_k\to\pi_1(X,z)$, which in turn induces an \'etale abelian section $s_x^{\ab}:G_k\rightarrow\pi_1(X,z)^{(\ab)}$.  Thus we have a map $X(k)\to H^1(G_k,TJ)$ defined by $x\mapsto [s_x^{\ab}]$, where $[s_x^{\ab}]$ denotes the $\pi_1(X_{\bar k}, \bar z)^{\ab}$-conjugacy class of sections containing $s_x^{\ab}$.  This map factors through $J(k)$, that is, it coincides with the composite map
\[\begin{tikzcd}[column sep=small]X(k)\arrow[hookrightarrow]{r}{\iota} &J(k) \arrow{r}{} &H^1(G_k,TJ).\end{tikzcd}\]
This is due to the isomorphism $\pi_1(X,z)^{(\ab)}\simeq\pi_1(J,z)$, via which $s_x^{\ab}$ corresponds to the section of $\pi_1(J,z)$ induced by functoriality from the $k$-rational point $\iota(x)\in J(k)$.

\begin{lemma}\label{kummerexactseq}
There is an exact sequence
\[\sexactab{\widehat{J(k)}}{}{H^1(G_k,TJ)}{}{TH^1(G_k,J)}\]
where $\widehat{J(k)}:=\varprojlim_N J(k)/NJ(k)$ (see Notation) and $TH^1(G_k,J)$ is the Tate module of the Galois cohomology group $H^1(G_k,J)$.
\end{lemma}

We shall refer to this sequence as the \emph{Kummer exact sequence}.  One can easily derive it from the Kummer exact sequences $\begin{tikzcd}[column sep=small]0 \arrow{r}{} &J(\bar k)[N] \arrow{r}{} &J(\bar k) \arrow{r}{N} &J(\bar k) \arrow{r}{} &0,\end{tikzcd}$ $N \geq 0$, during which one sees, in particular, that the map $J(k)\to H^1(G_k,TJ)$ factors through $\widehat{J(k)}$.  Thus we have a sequence of maps:
\begin{equation}\label{sabptth}
\begin{tikzcd}[column sep=small]X(k) \arrow[hookrightarrow]{r}{\iota} &J(k) \arrow{r}{} &\widehat{J(k)} \arrow[hookrightarrow]{r}{} &H^1(G_k,TJ).\end{tikzcd}
\end{equation}

\section{Specialisation of sections in a local setting}

\subsection{A specialisation homomorphism for absolute Galois groups}\label{spechom}

Let $R$ be a complete discrete valuation ring with uniformiser $\pi$, field of fractions $K$ and residue field $k:=R/\pi R$ of characteristic zero. Let $X$ be a flat, proper, smooth, geometrically connected \emph{relative curve} over $\Spec R$, and denote by $X_K:=X\times _{\Spec R}\Spec K$ its generic fibre and $X_k:=X\times _{\Spec R}\Spec k$ its special fibre.  Fix an algebraic closure $\ob{K}$ of $K$, and denote by $\ob{R}$ the integral closure of $R$ in $\ob{K}$, and by $\bar k$ the residue field of $\ob{R}$, which is an algebraic closure of $k$.

Let $\bar\xi_1:\Spec\Omega_1\rightarrow X_{\ob{K}}$ be a geometric point with image the generic point of $X_{\ob{K}}$, and similarly let $\bar\xi_2:\Spec\Omega_2\rightarrow X_{\bar k}$ be a geometric point with image the generic point of $X_{\bar k}$.  These induce geometric points of $X_K$ and $X_k$, which we denote by $\xi_1$ and $\xi_2$ respectively.

\begin{definition}\label{stilde}
For each closed point $x$ of $X_k^{\cl}$, fix a choice of closed point $y\in X_K^{\cl}$ which specialises to $x$ and whose residue field is the unique unramified extension of $K$ whose valuation ring has residue field $k(x)$ (such a point exists since $X \to \Spec R$ is smooth).  We define $\tilde S$ to be the set of these chosen closed points $y\in X_K^{\cl}$.  Thus, $\tilde S$ is a subset of $X^{\cl}_K$ in bijection with $X_k^{\cl}$.
We denote by $\tilde S_{\ob{K}}$ the inverse image of $\tilde S$ via the map $X_{\ob K}^{\cl}\to X^{\cl}$.
Thus, $\tilde S_{\ob{K}}$ is a subset of $X_{\ob{K}}^{\cl}$ in bijection with $X_{\bar k}^{\cl}$.
\end{definition}

\begin{definition}\label{pi1stilde}
With $\tilde S$ and $\tilde S_{\ob{K}}$ as in Definition \ref{stilde}, we define the group $\pi_1(X_K - \tilde S)$ to be the inverse limit
\[\pi_1(X_K - \tilde S) := \varprojlim_{B\subset\tilde S\textup{ finite}}\pi_1(X_K \setminus B, \xi_1)\]
over the open subsets of $X_K$ whose complements are finite subsets of $\tilde S$, ordered by inclusion.  Similarly, we define $\pi_1(X_{\ob{K}} - \tilde S_{\ob{K}}) := \varprojlim_{B\subset\tilde S_{\ob{K}}\textup{ finite}}\pi_1(X_{\ob{K}} \setminus B, \bar\xi_1)$.
\end{definition}

\begin{definition}\label{universalcovertilde}
A \emph{universal pro-\'etale cover} $\tilde X_{\tilde S} \to X_K - \tilde S$ is an inverse system of finite morphisms $\{Y_i\to X_K\}_{i\in I}$, with $Y_i$ smooth, corresponding to open subgroups of 
$\pi_1(X_K - \tilde S)$, such that for any finite morphism $Y \to X_K$, with $Y$ smooth, corresponding to an open subgroup of $\pi_1(X_K - \tilde S)$,
 there is an $X_K$-morphism $Y_i\to Y$ for some $i\in I$.
A (closed) \emph{point} $\tilde x$ of $\tilde X_{\tilde S}$ is a compatible system of (closed) points $\{y_i \in Y_i\}_{i\in I}$.
\end{definition}

For a universal pro-\'etale cover $\tilde X_{\tilde S} \to X_K - \tilde S$ and any closed point $\tilde x$ of $\tilde X_{\tilde S}$, we define decomposition and inertia subgroups $D_{\tilde x}, I_{\tilde x} \subset \pi_1(X_K - \tilde S)$ exactly as in Definition \ref{decompinertiapi1}, and they satisfy analogous properties.

\begin{theorem}\label{galspec}
There exists a surjective (continuous) homomorphism $\Sp : \pi_1(X_K-\tilde S) \twoheadrightarrow G_{X_k}$, an isomorphism $\ob{\Sp} : \pi_1(X_{\ob{K}} - \tilde S_{\ob{K}}) \simeq G_{X_{\bar k}}$, and a commutative diagram of exact sequences:
\begin{equation}\label{specialisationgalois}
\begin{tikzpicture}[descr/.style={fill=white}, baseline=(current  bounding  box.center)]
\matrix(m)[matrix of math nodes,
row sep=3em, column sep=2.5em,
text height=1.5ex, text depth=0.25ex]
{1 & G_{X_{\ob{K}}} & G_{X_K} & G_K & 1\\
1 & \pi_1(X_{\ob{K}}-\tilde S_{\ob{K}}) & \pi_1(X_K-\tilde S) & G_K & 1\\
1 & G_{X_{\bar k}} & G_{X_k} & G_k & 1\\};
\path[->,font=\scriptsize]
(m-1-1) edge node[above]{} (m-1-2);
\path[->,font=\scriptsize]
(m-1-2) edge node[above]{} (m-1-3);
\path[->>,font=\scriptsize]
(m-1-2) edge node[left]{} (m-2-2);
\path[->,font=\scriptsize]
(m-1-3) edge node[above]{} (m-1-4);
\path[->>,font=\scriptsize]
(m-1-3) edge node[left]{} (m-2-3);
\path[->,font=\scriptsize]
(m-1-4) edge node[above]{} (m-1-5);
\path[-,font=\scriptsize]
(m-1-4) edge[double distance=2pt] node[left]{} (m-2-4);
\path[->,font=\scriptsize]
(m-2-1) edge node[above]{} (m-2-2);
\path[->,font=\scriptsize]
(m-2-2) edge node[above]{} (m-2-3);
\path[->>,font=\scriptsize]
(m-2-2) edge node[left]{$\ob{\Sp}$} (m-3-2);
\path[->,font=\scriptsize]
(m-2-3) edge node[above]{} (m-2-4);
\path[->>,font=\scriptsize]
(m-2-3) edge node[left]{$\Sp$} (m-3-3);
\path[->,font=\scriptsize]
(m-2-4) edge node[above]{} (m-2-5);
\path[->>,font=\scriptsize]
(m-2-4) edge node[left]{} (m-3-4);
\path[->,font=\scriptsize]
(m-3-1) edge node[above]{} (m-3-2);
\path[->,font=\scriptsize]
(m-3-2) edge node[above]{} (m-3-3);
\path[->,font=\scriptsize]
(m-3-3) edge node[above]{} (m-3-4);
\path[->,font=\scriptsize]
(m-3-4) edge node[above]{} (m-3-5);
\end{tikzpicture}
\end{equation}
The homomorphism $\Sp$, resp. $\ob{\Sp}$, is defined only up to conjugation.
\end{theorem}

The homomorphisms $\ob{\Sp}$ and $\Sp$ will be referred to as \emph{specialisation homomorphisms}.  For the proof we need the following well-known result.

\begin{lemma}\label{specialisationaffine}
Let $S$ be a divisor on $X$ which is finite \'etale over $R$, and denote $U := X \setminus S$ which is an open sub-scheme of $X$.  Then there exists a surjective (continuous) homomorphism $\Sp_U : \pi_1(U_K, \xi_1) \twoheadrightarrow \pi_1(U_k, \xi_2)$ and an isomorphism $\ob{\Sp}_U : \pi_1(U_{\ob{K}}, \bar\xi_1) \simeq \pi_1(U_{\bar k}, \bar\xi_1)$ making the following diagram commutative.
\begin{equation}\label{specialisationaffine2}
\begin{tikzpicture}[descr/.style={fill=white}, baseline=(current bounding box.center)]
\matrix(m)[matrix of math nodes,
row sep=3em, column sep=2.5em,
text height=1.5ex, text depth=0.25ex]
{1 & \pi_1(U_{\ob{K}},\bar \xi_1) & \pi_1(U_K,\xi_1) & G_K & 1\\
1 & \pi_1(U_{\bar k},\bar \xi_2) & \pi_1(U_k,\xi_2) & G_k & 1\\};
\path[->,font=\scriptsize]
(m-1-1) edge node[above]{} (m-1-2);
\path[->,font=\scriptsize]
(m-1-2) edge node[above]{} (m-1-3);
\path[->,font=\scriptsize]
(m-1-2) edge node[left]{$\ob{\Sp}_U$} node[below,rotate=90]{$\sim$} (m-2-2);
\path[->,font=\scriptsize]
(m-1-3) edge node[above]{} (m-1-4);
\path[->>,font=\scriptsize]
(m-1-3) edge node[left]{$\Sp_U$} (m-2-3);
\path[->,font=\scriptsize]
(m-1-4) edge node[above]{} (m-1-5);
\path[->>,font=\scriptsize]
(m-1-4) edge node[left]{} (m-2-4);
\path[->,font=\scriptsize]
(m-2-1) edge node[above]{} (m-2-2);
\path[->,font=\scriptsize]
(m-2-2) edge node[above]{} (m-2-3);
\path[->,font=\scriptsize]
(m-2-3) edge node[above]{} (m-2-4);
\path[->,font=\scriptsize]
(m-2-4) edge node[above]{} (m-2-5);
\end{tikzpicture}
\end{equation}
The homomorphism $\ob{\Sp}_U$, respectively $\Sp_U$, is defined only up to inner automorphism of $\pi_1(U_{\ob{k}},\bar\xi_2)$, resp. $\pi_1(U,\xi_2)$.
\end{lemma}

The homomorphisms $\ob{\Sp}_U$ and $\Sp_U$ are called \emph{specialisation homomorphisms of fundamental groups}.  The homomorphism $\Sp_U$ is defined in a natural way and induces the homomorphism $\ob{\Sp}_U$.  For an exposition, in particular the fact that $\ob{\Sp}_U$ is an isomorphism, see \cite{Luminy}.


\begin{proof}[Proof of Theorem \ref{galspec}]
Let $B\subset\tilde S$ be a finite subset of $\tilde S$, viewed as a reduced closed sub-scheme of $X_K$, and let $\B$ denote its schematic closure in $X$.  By construction, $\B$ is a divisor on $X$ which is finite \'etale over $R$ such that $\B_K=B$.  Denoting $U := X - \B$, by Lemma \ref{specialisationaffine} there exist specialisation homomorphisms $\Sp_U : \pi_1(U_K, \xi_1) \twoheadrightarrow \pi_1(U_k, \xi_2)$ and $\ob{\Sp}_U : \pi_1(U_{\ob{K}}, \bar\xi_1) \simeq \pi_1(U_{\bar k}, \bar\xi_2)$.

Since there exist such homomorphisms for every finite subset of $\tilde S$, we have a compatible system of surjective homomorphisms $\{ \Sp_U \}$, resp. of isomorphisms $\{ \ob{\Sp}_U \}$ (the compatibility follows from the construction of these homomorphisms).  
By Lemma \ref{Gallim}, taking inverse limits gives rise respectively to the surjective homomorphism $\Sp : \pi_1(X_K-\tilde S)\twoheadrightarrow G_{X_k}$ and the isomorphism $\ob{\Sp} : \pi_1(X_{\ob{K}} - \tilde S_{\ob{K}}) \simeq G_{X_{\bar k}}$, and moreover $\pi_1(X_K-\tilde S)$ and $\pi_1(X_{\ob{K}} - \tilde S_{\ob{K}})$ are naturally quotients of $G_{X_K}$ and $G_{X_{\ob{K}}}$ respectively.  Thus we have the required homomorphisms in diagram (\ref{specialisationgalois}), and this diagram is clearly commutative.
\end{proof}

One could consider the composite homomorphism $G_{X_K} \twoheadrightarrow \pi_1(X_K-\tilde S) \twoheadrightarrow G_{X_k}$, and similarly for $G_{X_{\ob{K}}}$, to be a specialisation homomorphism for absolute Galois groups.  However, we will reserve the label `Sp' for the homomorphism $\pi_1(X_K-\tilde S) \twoheadrightarrow G_{X_k}$, since this will be important in the next section.

\subsection{Ramification of sections}\label{ramificationofsections}

We use the notation of \S \ref{spechom}, and assume further that the closed fibre $X_k$ is \emph{hyperbolic}.  With $S$ and $U$ as in Lemma \ref{specialisationaffine}, consider again diagram (\ref{specialisationaffine2}), and recall that the kernel of the projection $G_K\twoheadrightarrow G_k$ is the inertia group $I_K$ associated to the discrete valuation on $K$.

\begin{lemma}\label{sppullback}
\begin{enumerate}
\item The projection $\pi_1(U_K, \xi_1) \twoheadrightarrow G_K$ restricts to an isomorphism $\ker(\Sp_U)\simeq I_K$.
\item The right square in diagram (\ref{specialisationaffine2}) is cartesian.
\end{enumerate}
\end{lemma}

\begin{proof}
The isomorphism $\ker(\Sp)\simeq I_K$ follows from a simple diagram chase, and (ii) follows easily from (i).
\end{proof}

Fix universal pro-\'etale covers $\tilde U_K \to U_K$ and $\tilde U_k \to U_k$ (corresponding to the geometric points $\xi_1$ and $\xi_2$ respectively), and let $\tilde X_{\tilde U_K}$ denote the normalisation of $X_K$ in $\tilde U_K$, and likewise $\tilde X_{\tilde U_k}$ the normalisation of $X_k$ in $\tilde U_k$ (see Definitions \ref{universalcover} and \ref{normalisation}).

\begin{lemma}\label{cuspdecomppullback}
Let $x\in S(k)$, and let $y'$ be the unique ($K$-rational) point of $S_K$ which specialises to $x$.  Let $\tilde y'$ be a point of $\tilde X_{\tilde U_K}$ above $y'$.  There exists a unique $\tilde x$ in $\tilde X_{\tilde U_k}$ above $x$, so that we have the following commutative diagram:
\[\begin{tikzpicture}[descr/.style={fill=white}, baseline=(current bounding box.center)]
\matrix(m)[matrix of math nodes,
row sep=3em, column sep=2.5em,
text height=1.5ex, text depth=0.25ex]
{1 & I_{\tilde y'} & D_{\tilde y'} & G_K & 1\\
1 & I_{\tilde x} & D_{\tilde x} & G_k & 1\\};
\path[->,font=\scriptsize]
(m-1-1) edge node[above]{} (m-1-2);
\path[->,font=\scriptsize]
(m-1-2) edge node[above]{} (m-1-3);
\path[->,font=\scriptsize]
(m-1-2) edge node[left]{$\ob{\Sp}_U$} node[below,rotate=90]{$\sim$} (m-2-2);
\path[->,font=\scriptsize]
(m-1-3) edge node[above]{} (m-1-4);
\path[->>,font=\scriptsize]
(m-1-3) edge node[left]{$\Sp_U$} (m-2-3);
\path[->,font=\scriptsize]
(m-1-4) edge node[above]{} (m-1-5);
\path[->>,font=\scriptsize]
(m-1-4) edge node[left]{} (m-2-4);
\path[->,font=\scriptsize]
(m-2-1) edge node[above]{} (m-2-2);
\path[->,font=\scriptsize]
(m-2-2) edge node[above]{} (m-2-3);
\path[->,font=\scriptsize]
(m-2-3) edge node[above]{} (m-2-4);
\path[->,font=\scriptsize]
(m-2-4) edge node[above]{} (m-2-5);
\end{tikzpicture}\]
where $D_{\tilde y'}$ (resp. $D_{\tilde x}$) is the decomposition group of $\tilde y'$ (resp. $\tilde x$) in $\pi_1(U_K,\xi_1)$ (resp. $\pi_1(U_k,\xi_2)$).
Moreover, the right square in this diagram is cartesian.
\end{lemma}

\begin{proof}
The image of $D_{\tilde y'}$ under $\Sp_U$ is contained in $D_{\tilde x}$ for some $\tilde x$ in $\tilde X_{U_k}$ above $x$, as follows easily from the functoriality of fundamental groups and the specialisation of points on (coverings of) $R$-curves.  Moreover, such $\tilde x$ is unique by Lemma \ref{inertiaproperties}, since $\ob{\Sp}_U$ is an isomorphism, which implies the inertia subgroup $I_{\tilde y'}\subset D_{\tilde y'}$ maps isomorphically to $I_{\tilde x}\subset\pi_1(U_{\bar k},\bar\xi_2)$.  Commutativity of diagram (\ref{specialisationaffine2}) then implies that $D_{\tilde y'}$ maps surjectively onto $D_{\tilde x}$, whence the above diagram.  As in the proof of Lemma \ref{sppullback}, the right square in this diagram is cartesian.
\end{proof}

\begin{corollary}\label{cuspidalevenifyisnotacusp}
With the notation of Lemma \ref{cuspdecomppullback}, for any $K$-rational point $y$ of $U_K$ specialising to $x$ and any point $\tilde y$ of $\tilde U_K$ above $y$, $D_{\tilde y}$ is contained in $D_{\tilde y'}$ for some $\tilde y'$ in $\tilde X_{\tilde U_K}$ above $y'$.  In particular, a section $s:G_K\to\pi_1(U_K,\xi_1)$ with image contained in (hence equal to) $D_{\tilde y}$ is cuspidal (Definition \ref{cuspidal}), even though $y\not\in S_K$.
\end{corollary}

\begin{proof}
By specialisation, the point $\tilde y$ determines a point $\tilde x\in \tilde X_{\tilde U_k}$ such that $\Sp_U(D_{\tilde y}) \subset D_{\tilde x}$.  The statement then follows from Lemma \ref{cuspdecomppullback} and the universal property of cartesian squares. Note that for a point $\tilde x$ in $\tilde X_{\tilde U_k}$ above $x$ there exists a point $\tilde y'$ of $\tilde X_{\tilde U_K}$ above $y'$ such that the conclusion of Lemma \ref{cuspdecomppullback} holds (follows easily from a limit argument).
\end{proof}

By Lemma \ref{sppullback} (ii), any section of $\pi_1(U_k,\xi_2)$ naturally induces a section of $\pi_1(U_K,\xi_1)$, but the converse is not true in general.  Given any section $s:G_K\rightarrow\pi_1(U_K,\xi_1)$, let us define a homomorphism $\varphi_s:G_K\rightarrow\pi_1(U_k,\xi_2)$ by the composition $\varphi_s:=\Sp_U\circ s$.

\begin{equation}\label{defofvarphisandramified}
\begin{tikzpicture}[descr/.style={fill=white}, baseline=(current  bounding  box.center)]
\matrix(m)[matrix of math nodes,
row sep=4em, column sep=3em,
text height=1.5ex, text depth=0.25ex]
{1 & \pi_1(U_{\ob{K}},\bar\xi_1) & \pi_1(U_K, \xi_1) & G_K & 1\\
1 & \pi_1(U_{\bar k},\bar\xi_2) & \pi_1(U_k, \xi_2) & G_k & 1\\};
\path[->]
(m-1-1) edge (m-1-2);
\path[->,font=\scriptsize]
(m-1-2) edge (m-1-3) edge node[left]{$\ob{\Sp}_U$} node[below,rotate=90]{$\sim$} (m-2-2);
\path[->]
(m-1-3) edge (m-1-4);
\path[->>,font=\scriptsize]
(m-1-3) edge node[left]{$\Sp_U$} (m-2-3);
\path[->,font=\scriptsize]
(m-1-4) edge (m-1-5) edge[out=160,in=20] node[above]{$s$} (m-1-3) edge[out=200,in=60] node[left]{$\varphi_s$} (m-2-3);
\path[->>,font=\scriptsize]
(m-1-4) edge node[left]{} (m-2-4);
\path[->]
(m-2-1) edge (m-2-2);
\path[->]
(m-2-2) edge (m-2-3);
\path[->]
(m-2-3) edge (m-2-4);
\path[->]
(m-2-4) edge (m-2-5);
\end{tikzpicture}
\end{equation}

\begin{definition}\label{ramified}
We say the section $s$ is \emph{unramified} if $\varphi_s(I_K)=1$.  Otherwise we say $s$ is \emph{ramified}.

For an unramified section $s:G_K\rightarrow\pi_1(U_K,\xi_1)$ the map $\varphi_s$ factors through the projection $G_K\twoheadrightarrow G_k$ and $s$
induces a section $\bar s:G_k\rightarrow\pi_1(U_k,\xi_1)$.  This induced section will be called the \emph{specialisation of} $s$ and denoted $\bar s$.
\end{definition}

We now investigate under what conditions we may conclude that a given section $s : G_K \to \pi_1(U_K, \xi_1)$ has image contained in a decomposition group.

\begin{lemma}\label{unramifiedcuspidal}
Let $s:G_K\rightarrow\pi_1(U_K,\xi_1)$ be an unramified section, and suppose $\bar s(G_k)\subset D_{\tilde x}$ for some $x \in S(k)$ and some $\tilde x$ in $\tilde X_{\tilde U_k}$ above $x$.  Then $s(G_K)\subset D_{\tilde y'}$ for a unique $\tilde y'$ in $\tilde X_{\tilde U_K}$ above the unique point $y' \in S(K)$ specialising to $x$.
\end{lemma}

\begin{proof}
There exists a point $\tilde y'$ in $\tilde X_{\tilde U_K}$ above $y'$ such that the image of $D_{\tilde y'} \subset \pi_1(U_K, \xi_1)$ under $\Sp_U$ is contained in $D_{\tilde x}$ (cf. proof of Corollary \ref{cuspidalevenifyisnotacusp}), and this $\tilde y'$ is unique as follows from Lemma \ref{inertiaproperties}.  The pullback of $\bar s$ via the natural projection $G_K\twoheadrightarrow G_k$ gives rise to the section $s$, so $s(G_K)$ must be contained in the pullback of $D_{\tilde x}$ via $G_K\twoheadrightarrow G_k$, which is $D_{\tilde y'}$ by Lemma \ref{cuspdecomppullback}.
\end{proof}

\begin{proposition}\label{weightargument}
Assume that $k$ satisfies condition (ii) in Definition \ref{conditions}.  Let $s:G_K\rightarrow\pi_1(U_K,\xi_1)$ be a section, and denote $\varphi_s:=\Sp_U\circ s$ as above.  If $\varphi_s(I_K)$ is non-trivial then it is contained in the inertia group $I_{\tilde x}$ of a unique point $\tilde x$ of $\tilde X_{\tilde U_k}$ above a point of $S_k$.
\end{proposition}

\begin{proof}
This follows from \cite[Lemma 1.6]{hoshimochizuki}.  Indeed, if $\varphi_s(I_K)$ is non-trivial then, by commutativity of diagram (\ref{defofvarphisandramified}), it is contained in $\pi_1(U_{\bar k},\bar\xi_2)$, and it is a procyclic subgroup of $\pi_1(U_{\bar k},\bar\xi_2)$ because $I_K\simeq\hat\Z(1)$ is procyclic.  Since $k$ satisfies condition (ii) of Definition \ref{conditions}, the image of $\varphi_s(I_K)$ under the composite $G_k$-homomorphism
\[\begin{tikzcd}[column sep=small]
\pi_1(U_{\bar k},\bar\xi_2) \arrow[twoheadrightarrow]{r}{} &\pi_1(X_{\bar k},\bar\xi_2) \arrow[twoheadrightarrow]{r}{} &\pi_1(X_{\bar k},\bar\xi_2)^{\ab}
\end{tikzcd}\]
is trivial.  It then follows from loc. cit. that $\varphi_s(I_K)$ must be contained in an inertia subgroup of $\pi_1(U_{\bar k}, \bar\xi_2)$, which is unique by Lemma \ref{inertiaproperties}.
\end{proof}

\begin{remark}\label{properimpliesunramified}
Under the hypotheses of Proposition \ref{weightargument}, if $S=\emptyset$, so that $U=X$ is proper over $\Spec R$, then it follows from the arguments in the proof of Proposition \ref{weightargument} that any section $s:G_K\rightarrow\pi_1(X_K,\xi_1)$ is necessarily unramified.  
\end{remark}

\begin{lemma}\label{ramified=cuspidal}
Assume that $k$ satisfies condition (ii) in Definition \ref{conditions}, and let $s:G_K\to\pi_1(U_K,\xi_1)$ be a section and $\varphi_s:=\Sp_U\circ s$.  If $s$ is ramified then $\varphi_s(G_K)\subset D_{\tilde x}$ for a unique $k$-rational point $x$ of $S_k$ and a unique $\tilde x$ in $\tilde X_{\tilde U_k}$ above $x$, and $s(G_K)\subset D_{\tilde y'}$ for a unique point $\tilde y'$ of $\tilde X_{\tilde U_K}$ above the unique $K$-rational point $y'$ of $S_K$ specialising to $x$.  In particular, $s$ is cuspidal.
\end{lemma}

\begin{proof}
If $\varphi_s(I_K)$ is non-trivial then, by Proposition \ref{weightargument}, it must be contained in a unique inertia group $I_{\tilde x}$ for some $x\in S_k$ and some $\tilde x\in\tilde X_{\tilde U_k}$ above $x$.  Since $\varphi_s(G_K)$ normalises $\varphi_s(I_K)$, for some $\sigma\in G_K$ we have $\varphi_s(I_K)=\varphi_s(\sigma)\cdot \varphi_s(I_K)\cdot \varphi_s(\sigma)^{-1}\subseteq \varphi_s(\sigma)\cdot I_{\tilde x}\cdot \varphi_s(\sigma)^{-1}=I_{\varphi_s(\sigma)\cdot\tilde x}$.  But $\varphi_s(I_K)$ is contained in a unique inertia group, so $\varphi_s(\sigma)\cdot\tilde x=\tilde x$ and $\varphi_s(G_K)$ fixes $\tilde x$, i.e. $\varphi_s(G_K)\subseteq D_{\tilde x}$.  Moreover, $x$ is necessarily a $k$-rational point since, by commutativity of diagram (\ref{defofvarphisandramified}), $\varphi_s(G_K)$ maps surjectively onto $G_k$.  A similar argument to that used in the proof of Corollary 3.8 implies that $s(G_K)\subseteq D_{\tilde y'}$ for some $\tilde y'$ in $\tilde X_{\tilde U_K}$ above the unique point $y'$ of $S_K$ specialising to $x$. Moreover, such point $\tilde y'$ is unique by Lemma \ref{inertiaproperties}.
\end{proof}


Let $\tilde S$ be as in Definition \ref{stilde}.  For a section $s:G_K\rightarrow\pi_1(X_K-\tilde S)$ (see Definition \ref{pi1stilde}), we write $\varphi_s:=\Sp\circ s$ for the composition of $s$ with the specialisation homomorphism $\Sp : \pi_1(X_K - \tilde S)\twoheadrightarrow G_{X_k}$ of Theorem \ref{galspec}.
\begin{equation}\label{varphistilde}
\begin{tikzpicture}[descr/.style={fill=white}, baseline=(current  bounding  box.center)]
\matrix(m)[matrix of math nodes,
row sep=4em, column sep=3em,
text height=1.5ex, text depth=0.25ex]
{\pi_1(X_K-\tilde S) & G_K\\
G_{X_k} & G_k\\};
\path[->]
(m-1-1) edge (m-1-2);
\path[->>,font=\scriptsize]
(m-1-1) edge node[left]{$\Sp$} (m-2-1);
\path[->,font=\scriptsize]
(m-1-2) edge[out=150,in=30] node[above]{$s$} (m-1-1) edge[out=200,in=60] node[left]{$\varphi_s$} (m-2-1);
\path[->>,font=\scriptsize]
(m-1-2) edge node[left]{} (m-2-2);
\path[->,font=\scriptsize]
(m-2-1) edge node[above]{} (m-2-2);
\end{tikzpicture}
\end{equation}
We define the ramification and specialisation of $s$ analogously to Definition \ref{ramified}.  For an open subset $U_k \subset X_k$ with complement $S_k$, we will denote by $S_K$ the set of points of $\tilde S$ which specialise to $S_k$, $D$ the schematic closure of $S_K$ in $X$, and $U := X \setminus D$, thus $U_K=X_K\setminus S_K$. We will denote by $s_U : G_K\rightarrow\pi_1(U_K,\xi_1)$ the section of $\pi_1(U_K, \xi_1)$ naturally induced by $s$, and by $\varphi_U$ the composition $\varphi_U := \Sp_U \circ s_U : G_K \to \pi_1(U_k, \xi_2)$ of $s_U$ with the specialisation homomorphism $\Sp_U : \pi_1(U_K,\xi_1) \twoheadrightarrow \pi_1(U_k,\xi_2)$ of Lemma \ref{specialisationaffine}.

Since $s$ induces such a section $s_U : G_K \to \pi_1(U_K, \xi_1)$ for every open subset $U_K\subset X_K$ as above, it determines a compatible system of sections $\{ s_U \}$, parameterised by the open subsets $U\subset X$ as above.  Conversely, such a system determines a section of $\pi_1(X_K-\tilde S)$.  Similarly, since $s$ induces a homomorphism $\varphi_U : G_K \to \pi_1(U_k, \xi_2)$ for every open subset of $X_k$, it determines a compatible system of homomorphisms $\{ \varphi_U \}$ whose inverse limit $\varprojlim_{U_k \subset X_k \textup{ open}}\varphi_U$ is exactly the homomorphism $\varphi_s$ of diagram (\ref{varphistilde}).

\begin{lemma}\label{toramifiedsubsets}
A section $s:G_K\to\pi_1(X_K-\tilde S)$ is ramified if and only if there is a non empty open subset $U_k\subset X_k$ for which the section $s_U:G_K\to\pi_1(U_K,\xi_1)$ induced as above by $s$ is ramified.
\end{lemma}

\begin{proof}
Since $\varphi_s(I_K) = \varprojlim_{U_k \subset X_k \textup{ open}} \varphi_U(I_K)$ (with surjective transition maps), $\varphi_s(I_K)$ is trivial if and only if $\varphi_U(I_K)$ is trivial for every open subset $U_k \subset X_k$.
\end{proof}

\begin{proposition}\label{ramifiedsubsets}
Assume that $k$ satisfies condition (ii) of Definition \ref{conditions}, and let $s:G_K\to\pi_1(X_K-\tilde S)$ be a section.  Let $U'_k \subset U_k \subset X_k$ be any two non-empty open subsets, and let $\tilde X_{\tilde U_k}$, resp. $\tilde X_{\tilde U'_k}$ be the normalisation of $X_k$ in some universal pro-\'etale cover $\tilde U_k \to U_k$, resp. $\tilde U'_k \to U'_k$.

Suppose that $s_U$ is ramified, with $\varphi_U(I_K)$ contained in the inertia subgroup $I_{\tilde x_U} \subset \pi_1(U_k,\xi_2)$ for some $x\in (X_k\setminus U_k)(k)$ and some $\tilde x_U$ in $\tilde X_{\tilde U_k}$ above $x$ (see Lemma \ref{ramified=cuspidal} and its proof).  Then $s_{U'}$ is ramified, and $\varphi_{U'}(I_K)$ is contained in the inertia subgroup $I_{\tilde x_{U'}}\subset\pi_1(U'_k,\xi_2)$ of some $\tilde x_{U'}$ in $\tilde X_{\tilde U'_k}$ above the same point $x\in (X_k\setminus U_k)(k)\subset (X_k\setminus U_k')(k)$.
\end{proposition}

\begin{proof}
The image of $\varphi_{U'}(I_K)$ under the homomorphism $\pi_1(U'_k,\xi_2)\twoheadrightarrow\pi_1(U_k,\xi_2)$ coincides with $\varphi_U(I_K)$, which is nontrivial by assumption.  Thus $\varphi_{U'}(I_K)$ must also be non-trivial, and therefore it is contained in an inertia subgroup $I_{\tilde z_{U'}} \subset \pi_1(U'_k,\xi_2)$ for some $z\in (X_k\setminus U_k')(k)$ and some $\tilde z_{U'}$ in $\tilde X_{\tilde U'_k}$ above $z$ (see Lemma \ref{ramified=cuspidal} and its proof).  Let $\tilde z_U$ be the image of $\tilde z_{U'}$ in $\tilde X_{\tilde U_k}$.  Suppose $\tilde z_U \ne \tilde x_U$.  If $z \in X_k\setminus U_k$ then, by functoriality, the image of $I_{\tilde z_{U'}}$ under $\pi_1(U'_k,\xi_2)\twoheadrightarrow\pi_1(U_k,\xi_2)$ is the inertia subgroup $I_{\tilde z_U} \subset \pi_1(U_k,\xi_2)$ at $\tilde z_U$, which intersects trivially with $I_{\tilde x_U}$ by Lemma \ref{inertiaproperties}.  Meanwhile, if $z \not\in X_k\setminus U_k$ the image of $I_{\tilde z_{U'}}$ in $\pi_1(U_k,\xi_2)$ is trivial.  Both of these contradict compatibility of $\varphi_U$ and $\varphi_{U'}$, so we must have $\tilde z_U=\tilde x_U$ and $z=x$.
\end{proof}

The ramification of a section $s:G_K\to\pi_1(X_K-\tilde S)$ is therefore characterised by the ramification of the system of sections $s_U$ it induces.  Let us now fix a universal pro-\'etale cover $\tilde X_{\tilde S} \to X_K - \tilde S$ (see Definition \ref{universalcovertilde}), and denote by $\overline {k(X_k)}$ the separable closure of the function field $k(X_k)$ determined by the geometric point $\xi_2$.

\begin{lemma}\label{galptthunramified}
Let $s:G_K\to\pi_1(X_K-\tilde S)$ be an unramified section, and suppose its specialisation $\bar s : G_k \to G_{X_k}$ is geometric with $\bar s(G_k)\subset D_{\tilde x}$ for some $x\in X(k)$ and some extension $\tilde x$ of $x$ to $\overline {k(X_k)}$.  Then $s(G_K)\subset D_{\tilde y}$ for a unique $\tilde y$ in $\tilde X_{\tilde S}$ above the unique ($K$-rational) point $y$ of $\tilde S$ specialising to $x$.
\end{lemma}

\begin{proof}
For any open subset $U_k \subset X_k$, the section $s_U$ is unramified by Lemma \ref{toramifiedsubsets}.  Choose $U_k$ so that $x \not\in U_k$, and let $\tilde X_{\tilde U_K}$, respectively $\tilde X_{\tilde U_k}$ denote the normalisation of $X_K$, resp. $X_k$ in some universal pro-\'etale cover of $U_K$, resp. $U_k$.  By compatibility of the homomorphisms $\varphi_U$, we have $\bar s_U(G_k) \subset D_{\tilde x_U}$ for some $\tilde x_U$ in $\tilde X_{U_k}$ above $x$.  Hence, by Lemma \ref{unramifiedcuspidal}, $s_U(G_K) \subset D_{\tilde y_U}$ for some unique $\tilde y_U$ in $\tilde X_{U_K}$ above the unique point $y$ of $S_K$ specialising to $x$.  This is also true for every open subset of $X_k$ contained in $U_k$, so taking inverse limits yields the required statement.
\end{proof}

\begin{lemma}\label{galptthramified}
Assume that $k$ satisfies condition (ii) of Definition \ref{conditions}, and let $s:G_K\to\pi_1(X_K-\tilde S)$ be a section.  If $s$ is ramified then $\varphi_s(G_K)\subset D_{\tilde x}$ for a unique valuation $\tilde x$ on $\overline {k(X_k)}$ extending a $k$-rational point $x$ of $X_k$, and $s(G_K)\subset D_{\tilde y}$ for a unique $\tilde y$ in $\tilde X_{\tilde S}$ above the unique $K$-rational point $y$ in $\tilde S$ specialising to $x$.
\end{lemma}

\begin{proof}
Lemma \ref{toramifiedsubsets} implies that, for some open subset $U_k \subset X_k$, the section $s_U$ is ramified.  Let $\tilde X_{\tilde U_K}$, respectively $\tilde X_{\tilde U_k}$ denote the normalisation of $X_K$, resp. $X_k$ in some universal pro-\'etale cover of $U_K$, resp. $U_k$.  By Lemma \ref{ramified=cuspidal} we have $\varphi_U(G_K) \subset D_{\tilde x_U}$ for some $x \in (X_k\setminus U_k)(k)$ and some $\tilde x_U$ in $\tilde X_{\tilde U_k}$ above $x$, and $s_U(G_K) \subset D_{\tilde y_{U}}$ for some unique $\tilde y_U$ in $\tilde X_{\tilde U_K}$ above the unique $K$-rational point $y$ of $\tilde S$ specialising to $x$.  By Proposition \ref{ramifiedsubsets}, this is true for every open subset $U'_k\subset X_k$ contained in $U_k$, thus taking inverse limits over the open subsets of $X_k$ yields the required statement (the uniqueness of $\tilde x$ follows from the same argument as in the proof of \cite[Corollary 12.1.3]{CONF}).
\end{proof}

\section{Sections of absolute Galois groups of curves over function fields}

\subsection{Specialisation of sections of absolute Galois groups}\label{specialisationofsections}

Let $k$ be a field of characteristic zero, and let $C$ a smooth, separated, connected curve over $k$ with function field $K$.  Let $\X\rightarrow C$ be a flat, proper, smooth relative curve, with generic fibre $X := \X \times_C K$ which is a geometrically connected curve over $K$.  For some $c\in C^{\cl}$, let $K_c$ denote the completion of $K$ with respect to the valuation corresponding to $c$, and let $X_c:=X \times_K K_c$ be the base change of $X$ to $K_c$.  

Let $\xi$ be a geometric point of $X$ with value in its generic point.  This determines an algebraic closure $\overline {K(X)}$ of the function field of $X$ and an algebraic closure $\ob{K}$ of $K$, as well as a geometric point $\bar\xi$ of $X_{\ob{K}}$.  Likewise let $\xi_c$ be a geometric point of $X_c$ with value in its generic point, which determines an algebraic closure $\ob{K_c}$ of $K_c$ and a geometric point $\bar\xi_c$ of $X_{\ob{K_c}}$.  Fix an embedding $\ob{K}\hookrightarrow \ob{K_c}$.  This determines a natural inclusion $i : X(\ob{K}) \hookrightarrow X_c(\ob{K_c})$.  Set-theoretically, an element $y : \Spec\ob{K_c} \to X_c$ of $X_c(\ob{K_c})$ maps the unique point of $\Spec\ob{K_c}$ to a closed point of $X_c$.  This closed point will be called an \emph{algebraic point} of $X_{K_c}$ if $y \in i(X(\ob{K}))$; otherwise, it will be called a \emph{transcendental point}.  Let us denote by $X_c^{\tr}$ the complement in $X_c^{\cl}$ of the set of all algebraic points of $X_c^{\cl}$.

\begin{definition}
With the above notation, define the group $\pi_1(X_c^{\tr})$ to be the inverse limit
\[\pi_1(X_c^{\tr}) := \varprojlim_{U \subset X \textup{ open}} \pi_1(U_c, \xi_c)\]
over all open subsets $U \subset X$, where $U_c$ denotes the base change $U \times_K K_c$.
\end{definition}

Recall there is an exact sequence $1 \to G_{X_{\ob{K}}} \to G_X \to G_K := \Gal(\ob{K}/K) \to 1$, where $G_X$ is the Galois group of $X$ with base point $\xi$ (cf. Lemma \ref{Gallim}).

\begin{lemma}\label{pullbackgx}
We have a commutative diagram of exact sequences
\[\begin{tikzpicture}[descr/.style={fill=white}, baseline=(current  bounding  box.center)]
\matrix(m)[matrix of math nodes,
row sep=3em, column sep=2.5em,
text height=1.5ex, text depth=0.25ex]
{1 & \pi_1(X_{\ob{K_c}}^{\tr}) & \pi_1(X_c^{\tr}) & G_{K_c} & 1\\
1 & G_{X_{\ob{K}}} & G_X & G_K & 1\\};
\path[->]
(m-1-1) edge (m-1-2);
\path[->]
(m-1-2) edge (m-1-3) edge (m-2-2);
\path[->]
(m-1-3) edge (m-1-4) edge (m-2-3);
\path[->]
(m-1-4) edge (m-1-5) edge (m-2-4);
\path[->]
(m-2-1) edge (m-2-2);
\path[->]
(m-2-2) edge (m-2-3);
\path[->]
(m-2-3) edge (m-2-4);
\path[->]
(m-2-4) edge (m-2-5);
\end{tikzpicture}\]
where $G_{K_c}=\Gal (\ob{K_c}/K_c)$, $\pi_1(X_{\ob{K_c}}^{\tr})$ is defined so that the upper horizontal sequence is exact, the middle vertical map is defined up to conjugation,
the left vertical map is an isomorphism, and the right square is cartesian.
\end{lemma}

\begin{proof}
For each open subset $U \subset X$, functoriality of the fundamental group yields a diagram
\begin{equation}\label{pullbackdiagrampi1u}
\begin{tikzpicture}[descr/.style={fill=white}, baseline=(current  bounding  box.center)]
\matrix(m)[matrix of math nodes,
row sep=3em, column sep=2.5em,
text height=1.5ex, text depth=0.25ex]
{1 & \pi_1(U_{\ob{K_c}},\bar\xi_c) & \pi_1(U_c,\xi_c) & G_{K_c} & 1\\
1 & \pi_1(U_{\ob{K}},\bar\xi) & \pi_1(U,\xi) & G_K & 1\\};
\path[->]
(m-1-1) edge (m-1-2);
\path[->]
(m-1-2) edge (m-1-3) edge (m-2-2);
\path[->]
(m-1-3) edge (m-1-4) edge (m-2-3);
\path[->]
(m-1-4) edge (m-1-5) edge (m-2-4);
\path[->]
(m-2-1) edge (m-2-2);
\path[->]
(m-2-2) edge (m-2-3);
\path[->]
(m-2-3) edge (m-2-4);
\path[->]
(m-2-4) edge (m-2-5);
\end{tikzpicture}
\end{equation}
where the rows are the fundamental exact sequences, the middle vertical map is defined up to conjugation,
the left vertical map is an isomorphism (see \cite{Luminy}, Th\'eor\`eme 1.6),
and the right square is cartesian (follows as in the proof of Lemma \ref{sppullback}).  The statement then follows by taking the projective limit of these diagrams over the open subsets of $X$.
\end{proof}


Denote by $\X_c:=\X\times_C k(c)$ the fibre of $\X$ above $c$.  


\begin{lemma}\label{liftingtoalgpts}
For each closed point $x$ of $\X_c^{\cl}$, there exists an algebraic point $y$ of $X_c$ specialising to $x$ whose residue field is the unique unramified extension $L$ of $K_c$ whose valuation ring $\oh_{L}$ has residue field $k(x)$.
\end{lemma}

\begin{proof}
Let us write $X_{c,L} := X_c \times_{K_c} L$ and $\X_{c,k(x)} := \X_c \times_{k(c)} k(x)$.  Let $x'$ be a $k(x)$-rational point of $\X_{c,k(x)}$ which maps to $x$ under the projection $\X_{c,k(x)} \to \X_c$, and let $\m_{L}$ denote the maximal ideal of $\oh_{L}$.  The set of $L$-rational points of $X_{c,L}$ which specialise to $x'$ is in bijection with $\m_{L}$ \cite[Proposition 10.1.40]{Liu}.  Let $F$ be a finite extension of $K$ whose completion at a place above $c$ is $L$.  Then an element of $\m_{L} \cap F$ corresponds to an $L$-rational algebraic point $y'$ of $X_{c,L}$ which specialises to $x'$.  The image $y$ of $y'$ under the projection $X_{c,L} \to X_c$ is then an $L$-rational algebraic point of $X_c$ ($k(y)=L$) which specialises to $x$.
\end{proof}

\begin{definition}\label{stildekc}
For each closed point $x$ of $\X_c^{\cl}$, fix a choice of algebraic point $y\in X_c^{\cl}$ which specialises to $x$ and whose residue field is the unique unramified extension of $K_c$ whose valuation ring has residue field $k(x)$ (such a point exists by Lemma \ref{liftingtoalgpts}).  We define $\tilde S_c$ to be the set of these chosen algebraic points $y\in X_c^{\cl}$.  Thus, $\tilde S_c$ is a subset of $X_c^{\cl}$ which consists of algebraic points and which is in bijection with $\X_c^{\cl}$.
We denote by $\tilde S_{c,\ob{K_c}}$ the set of points of $X_c\times _{K_c}\ob {K_c}$ which map to points in $\tilde S_c$.  
Thus $\tilde S_{c,\ob{K_c}}$ is a subset of $X_{\ob{K_c}}^{\cl}$ in bijection with $\X_{\ob{k(c)}}^{\cl}$.  
\end{definition}

Let $\xi'_c$ be a geometric point of $\X_c$ with value in its generic point, which determines an algebraic closure $\ob{k(c)}$ of $k(c)$ and a geometric point $\bar\xi'_c$ of $\X_{\ob{k(c)}}$.  Let $1 \to  G_{\X_{\ob{k(c)}}} \to G_{\X_c} \to  G_{k(c)} \to 1$ be the exact sequence of the absolute Galois group of $\X_c$ with base point $\xi'_c$.

\begin{theorem}\label{galspeckc}
With $\tilde S_c$ and $\tilde S_{c,\ob{K_c}}$ as in Definition \ref{stildekc}, there exists a surjective homomorphism $\Sp : \pi_1(X_c - \tilde S_c) \twoheadrightarrow G_{\X_c}$, an isomorphism $\ob{\Sp} : \pi_1(X_{\ob{K_c}} - \tilde S_{c,\ob{K_c}}) \simeq G_{\X_{\ob{k(c)}}}$ and a commutative diagram of exact sequences:
\[\begin{tikzpicture}[descr/.style={fill=white}, baseline=(current bounding box.center)]
\matrix(m)[matrix of math nodes,
row sep=3em, column sep=2.5em,
text height=1.5ex, text depth=0.25ex]
{1 & \pi_1(X_{\ob{K_c}}^{\tr}) & \pi_1(X_c^{\tr}) & G_{K_c} & 1\\
1 & \pi_1(X_{\ob{K_c}}-\tilde S_{c,\ob{K_c}}) & \pi_1(X_c-\tilde S_c) & G_{K_c} & 1\\
1 & G_{\X_{\ob{k(c)}}} & G_{\X_c} & G_{k(c)} & 1\\};
\path[->]
(m-1-1) edge (m-1-2);
\path[->]
(m-1-2) edge (m-1-3);
\path[->]
(m-1-3) edge (m-1-4);
\path[->]
(m-1-4) edge (m-1-5);
\path[->>]
(m-1-2) edge (m-2-2);
\path[->>]
(m-1-3) edge (m-2-3);
\path[-]
(m-1-4) edge[double distance=2pt] (m-2-4);
\path[->]
(m-2-1) edge (m-2-2);
\path[->]
(m-2-2) edge (m-2-3);
\path[->]
(m-2-3) edge (m-2-4);
\path[->]
(m-2-4) edge (m-2-5);
\path[->,font=\scriptsize]
(m-2-2) edge node[left]{$\ob{\Sp}$} node[below,rotate=90]{$\sim$} (m-3-2);
\path[->>,font=\scriptsize]
(m-2-3) edge node[left]{$\Sp$} (m-3-3);
\path[->>,font=\scriptsize]
(m-2-4) edge node[left]{} (m-3-4);
\path[->]
(m-3-1) edge (m-3-2);
\path[->]
(m-3-2) edge (m-3-3);
\path[->]
(m-3-3) edge (m-3-4);
\path[->]
(m-3-4) edge (m-3-5);
\end{tikzpicture}\]
where the map $\Sp$ is defined up to conjugation.
\end{theorem}

\begin{proof}
The existence of $\ob{\Sp}$ and $\Sp$ follows from Theorem \ref{galspec}.  Since we have chosen $\tilde S_c$ to consist of algebraic points, $\pi_1(X_c-\tilde S_c)$, respectively $\pi_1(X_{\ob{K_c}} - \tilde S_{c,\ob{K_c}})$ is naturally a quotient of $\pi_1(X_c^{\tr})$, resp. $\pi_1(X_{\ob{K_c}}^{\tr})$.  Thus we have the required homomorphisms in the diagram, and it is clearly commutative.
\end{proof}

Let $s:G_K\rightarrow G_X$ be a section, and let $s_c:G_{K_c}\rightarrow\pi_1(X_c^{\tr})$ be the section of $\pi_1(X_c^{\tr})$ induced from $s$ (cf. Lemma \ref{pullbackgx}).  This naturally induces a section $\tilde s_c:G_{K_c}\rightarrow\pi_1(X_c-\tilde S_c)$ of the projection $\pi_1(X_c-\tilde S_c)\twoheadrightarrow G_{K_c}$. 

\begin{theorem}\label{stildeptth}
With the notation of the above paragraph, assume further that $X$ is hyperbolic.  Let $\tilde X_{c, \tilde S_c} \to X_c - \tilde S_c$ be a universal pro-\'etale cover (Definition \ref{universalcovertilde}), and denote by $\overline {k(\X_c)}$ the separable closure of the function field of $\X_c$ determined by the geometric point $\xi'_c$.
\begin{enumerate}
\item Suppose $\tilde s_c$ is unramified and induces a section $\bar s_c:G_{k(c)}\rightarrow G_{\X_c}$.  If $\bar s_c$ is geometric with $\bar s_c(G_{k(c)})\subset D_{\tilde x}$ for some $x\in\X_c(k(c))$ and some extension $\tilde x$ of $x$ to $\overline {k(\X_c)}$, then $\tilde s_c(G_{K_c})\subset D_{\tilde y}$ for some unique $\tilde y$ in $\tilde X_{c, \tilde S_c}$ above the unique ($K_c$-rational) point $y$ of $\tilde S_c$ specialising to $x$.
\item Assume further that $k$ strongly satisfies condition (ii) of Definition \ref{conditions}.  Suppose $\tilde s_c$ is ramified, and denote $\varphi_s:=\Sp\circ\tilde s_c$.  Then $\varphi_s(G_{K_c})\subset D_{\tilde x}$ for a unique valuation on $\overline {k(\X_c)}$ extending a $k(c)$-rational point $x$ of $\X_c$, and $\tilde s_c(G_{K_c})\subset D_{\tilde y}$ for some unique $\tilde y$ in $\tilde X_{c, \tilde S_c}$ above the unique ($K_c$-rational) point $y$ of $\tilde S_c$ specialising to $x$.
\end{enumerate}
\end{theorem}

\begin{proof}
Part (i) follows from Lemma \ref{galptthunramified}, and (ii) follows from Lemma \ref{galptthramified}.
\end{proof}

\subsection{\'Etale abelian sections}\label{etaleabeliansections}

We use the notation of \S \ref{specialisationofsections}, and hereafter we assume that $X$ is hyperbolic and that $X(K)\ne\emptyset$.

\begin{lemma}
For each $c \in C^{\cl}$, there is a commutative diagram
\begin{equation}\label{specialisationofset}
\begin{tikzpicture}[descr/.style={fill=white}, baseline=(current  bounding  box.center)]
\matrix(m)[matrix of math nodes,
row sep=3em, column sep=2.5em,
text height=1.5ex, text depth=0.25ex]
{1 & \pi_1(X_{\ob{K}}, \bar\xi) & \pi_1(X, \xi) & G_K & 1\\
1 & \pi_1(X_{\ob{K}}, \bar\xi) & \pi_1(\X, \xi) & \pi_1(C, \xi) & 1\\
1 & \pi_1(\X_{\ob{k(c)}}, \bar\xi'_c) & \pi_1(\X_c, \xi'_c) & G_{k(c)} & 1\\};
\path[->]
(m-1-1) edge (m-1-2);
\path[->]
(m-1-2) edge (m-1-3);
\path[->]
(m-1-3) edge (m-1-4);
\path[->, font = \scriptsize]
(m-1-4) edge (m-1-5);
\path[-]
(m-1-2) edge[double distance = 2pt] (m-2-2);
\path[->>]
(m-1-3) edge (m-2-3);
\path[->>]
(m-1-4) edge (m-2-4);
\path[->]
(m-2-1) edge (m-2-2);
\path[->]
(m-2-2) edge (m-2-3);
\path[->]
(m-2-3) edge (m-2-4);
\path[->, font = \scriptsize]
(m-2-4) edge (m-2-5);
\path[->]
(m-3-1) edge (m-3-2);
\path[->, font = \scriptsize]
(m-3-2) edge (m-3-3) edge node[below, rotate = 90]{$\sim$} (m-2-2);
\path[->]
(m-3-3) edge (m-3-4);
\path[right hook->]
(m-3-3) edge (m-2-3);
\path[->, font = \scriptsize]
(m-3-4) edge (m-3-5);
\path[right hook->]
(m-3-4) edge (m-2-4);
\end{tikzpicture}
\end{equation}
where the lower vertical homomorphisms are defined up to conjugation, the middle and right of those maps are injective and the left one is an isomorphism.
\end{lemma}

\begin{proof}
This follows from functoriality of the fundamental group and the fundamental exact sequences for $X$ and $\X_c$.  Exactness of the middle row follows from \cite[Expos\'e XIII, Proposition 4.3]{SGA}.
\end{proof}

For the remainder of this section we assume that $k$ strongly satisfies the condition (ii) in Definition 1.4.
Let us fix a section $s:G_K\rightarrow G_X$ of $G_X$, and denote by $s^{\et}:G_K\rightarrow\pi_1(X,\xi)$ the section of the \'etale fundamental group of $X$ induced by $s$.  By Lemma \ref{pullbackgx}, $s$ pulls back to a section $s_c:G_{K_c}\rightarrow\pi_1(X_c^{\tr})$, which in turn induces a section $s_c^{\et}:G_{K_c}\rightarrow\pi_1(X_c,\xi_c)$.  By Proposition \ref{weightargument} (see also Remark \ref{properimpliesunramified}), $s_c^{\et}$ specialises to a section $\bar s_c^{\et}:G_{k(c)}\rightarrow\pi_1(\X_c,\xi'_c)$.  By the following, we may also consider $\bar s_c^{\et}$ the specialisation of $s^{\et}$.

\begin{lemma}
The section $s^{\et}$ extends to a section $s_C^{\et} : \pi_1(C, \xi) \to \pi_1(\X, \xi)$ of the projection $\pi_1(\X, \xi) \twoheadrightarrow \pi_1(C, \xi)$ which restricts to the section $\bar s_c^{\et}$ for each $c \in C^{\cl}$.
\end{lemma}

\begin{proof}
The kernel of the homomorphism $G_K\twoheadrightarrow\pi_1(C, \xi)$ is the inertia group $I_C$ normally generated by the inertia subgroups associated to the closed points of $C$.  Since $s_c^{\et}$ is unramified for every $c \in C^{\cl}$, the image of each of these inertia groups under the composition $\Sp_X \circ s_c^{\et} : G_{K_c} \to \pi_1(\X_c, \xi'_c)$ of $s_c^{\et}$ with the specialisation homomorphism $\Sp_X : \pi_1(X_c, \xi_c) \twoheadrightarrow \pi_1(\X_c, \xi'_c)$ is trivial, hence the image of $I_C$ under the composite $\begin{tikzcd}[column sep=small]G_C \arrow{r}{s^{\et}} & \pi_1(X, \xi) \arrow[twoheadrightarrow]{r}{} & \pi_1(\X, \xi)\end{tikzcd}$ is trivial.  Thus $s^{\et}$ extends to a section $s_C^{\et}:\pi_1(C, \xi) \rightarrow \pi_1(\X, \xi)$, which must restrict to $\bar s_c^{\et}:G_{k(c)}\rightarrow\pi_1(\X_c, \xi'_c)$.
\end{proof}

Let $\J:=\textup{Pic}^0_{\X/C}\rightarrow C$ denote the relative Jacobian of $\X$, and $J:=\J_K$ the Jacobian of $X$.  For each closed point $c\in C^{\cl}$, let $J_c:=J_{K_c}$ denote the Jacobian of $X_c$ and $\J_c:=\J_{k(c)}$ that of $\X_c$.  The above sections $s^{\et}$, $s_c^{\et}$ and $\bar s_c^{\et}$ induce \'etale abelian sections $s^{\ab}$, $s_c^{\ab}$ and $\bar s_c^{\ab}$ respectively, while diagram (\ref{specialisationofset}) induces a commutative diagram of exact sequences of geometrically abelian fundamental groups
\begin{equation}\label{specialisationofsab}
\begin{tikzpicture}[descr/.style={fill=white}, baseline=(current  bounding  box.center)]
\matrix(m)[matrix of math nodes,
row sep=3em, column sep=2.5em,
text height=1.5ex, text depth=0.25ex]
{1 & \pi_1(X_{\ob{K}}, \bar\xi)^{\ab} & \pi_1(X, \xi)^{(\ab)} & G_K = G_C & 1\\
1 & \pi_1(X_{\ob{K}}, \bar\xi)^{\ab} & \pi_1(\X, \xi)^{(\ab)} & \pi_1(C, \xi) & 1\\
1 & \pi_1(\X_{\ob{k(c)}}, \bar\xi'_c)^{\ab} & \pi_1(\X_c, \xi'_c)^{(\ab)} & G_{k(c)} & 1\\};
\path[->]
(m-1-1) edge (m-1-2);
\path[->]
(m-1-2) edge (m-1-3);
\path[->]
(m-1-3) edge (m-1-4);
\path[->, font = \scriptsize]
(m-1-4) edge (m-1-5) edge[out = 165, in = 15] node[above]{$s^{\ab}$} (m-1-3);
\path[-]
(m-1-2) edge[double distance = 2pt] (m-2-2);
\path[->>]
(m-1-3) edge (m-2-3);
\path[->>]
(m-1-4) edge (m-2-4);
\path[->]
(m-2-1) edge (m-2-2);
\path[->]
(m-2-2) edge (m-2-3);
\path[->]
(m-2-3) edge (m-2-4);
\path[->, font = \scriptsize]
(m-2-4) edge (m-2-5) edge[out = 165, in = 15] node[above]{$s_C^{\ab}$} (m-2-3);
\path[->]
(m-3-1) edge (m-3-2);
\path[->, font = \scriptsize]
(m-3-2) edge (m-3-3) edge node[below, rotate = 90]{$\sim$} (m-2-2);
\path[->]
(m-3-3) edge (m-3-4);
\path[right hook->]
(m-3-3) edge (m-2-3);
\path[->, font = \scriptsize]
(m-3-4) edge (m-3-5) edge[out = 160, in = 15] node[above]{$\bar s_c^{\ab}$} (m-3-3);
\path[right hook->]
(m-3-4) edge (m-2-4);
\end{tikzpicture}
\end{equation}
where the middle horizontal row is obtained as the push-out of the middle horizontal row in diagram (9) by the projection 
$\pi_1(X_{\ob{K}},\bar\xi)\twoheadrightarrow \pi_1(X_{\ob{K}},\bar \xi)^{\ab}$,
and $s_C^{\ab} : \pi_1(C, \xi) \to \pi_1(\X, \xi)^{(\ab)}$ is induced by $s_C^{\et}$.  Since $X(K) \ne \emptyset$ by assumption, the \'etale abelian sections $s^{\ab}$, $s_c^{\ab}$, $s_C^{\ab}$ and $\bar s_c^{\ab}$ correspond to elements of the cohomology groups $H^1(G_K,TJ)$, $H^1(G_{K_c},TJ_c)$, $H^1(\pi_1(C, \xi), TJ)$ and $H^1(G_{k(c)},T\J_c)$ respectively, which are related by the following restriction and inflation maps:

\[\begin{tikzpicture}[descr/.style={fill=white}]
\matrix(m)[matrix of math nodes,
row sep=3em, column sep=2.5em,
text height=2ex, text depth=0.25ex]
{H^1(G_K,TJ) & H^1(G_{K_c},TJ_c)\\
H^1(\pi_1(C, \xi), TJ) & H^1(G_{k(c)},T\J_c)\\};
\path[->,font=\scriptsize]
(m-1-1) edge node[above]{$\res_c$} (m-1-2);
\path[->, font = \scriptsize]
(m-2-1) edge node[left]{$\inf_C$} (m-1-1) edge node[above]{$\res_{C, c}$} (m-2-2);
\path[->, font=\scriptsize]
(m-2-2) edge node[right]{$\inf_c$} (m-1-2);
\end{tikzpicture}\]

\begin{lemma}\label{sabresinf}
With the above notation, we have the following.
\begin{enumerate}
\item $\res_c(s^{\ab}) = \inf_c(\bar s_c^{\ab}) = s_c^{\ab}$;
\item $\inf_C(s_C^{\ab}) = s^{\ab}$ and $\res_{C, c}(s_C^{\ab}) = \bar s_c^{\ab}$.
\end{enumerate}
\end{lemma}

\begin{proof}
Part (i) follows from the fact that $s^{\et}$ and $\bar s^{\et}_c$ both pull back to $s^{\et}_c$ - see diagrams (\ref{pullbackdiagrampi1u}) and (\ref{defofvarphisandramified}), considering the case when $U=X$ is projective.  Part (ii) follows from diagram (\ref{specialisationofsab}).  See also \cite[Lemma 3.4]{saidiSCOFF}.
\end{proof}

The image of $s^{\ab}$ under the diagonal map
\[\prod_{c\in C^{\cl}}\res_c:H^1(G_K,TJ)\longrightarrow\prod_{c\in C^{\cl}}H^1(G_{K_c},TJ_c)\]
is therefore the family $(s_c^{\ab})_{c\in C^{\cl}}$.  This diagonal map fits into the following commutative diagram of Kummer exact sequences (see Lemma \ref{kummerexactseq}).
\[\begin{tikzpicture}
\matrix(m)[matrix of math nodes,
row sep=3em, column sep=2em,
text height=2ex, text depth=0.25ex]
{0 & \widehat{J(K)} & H^1(G_K,TJ) & TH^1(G_K,J) & 0\\
0 & \displaystyle\prod_{c\in C^{\cl}}\widehat{J_c(K_c)} & \displaystyle\prod_{c\in C^{\cl}}H^1(G_{K_c},TJ_c) & \displaystyle\prod_{c\in C^{\cl}}TH^1(G_{K_c},J_c) & 0\\};
\path[->]
(m-1-1) edge (m-1-2);
\path[->]
(m-1-2) edge (m-1-3) edge (m-2-2);
\path[->]
(m-1-3) edge (m-1-4) edge (m-2-3);
\path[font=\scriptsize,->]
(m-1-4) edge (m-1-5) edge node[left]{} (m-2-4);
\path[->]
(m-2-1) edge (m-2-2);
\path[->]
(m-2-2) edge (m-2-3);
\path[->]
(m-2-3) edge (m-2-4);
\path[->]
(m-2-4) edge (m-2-5);
\end{tikzpicture}\]
Note that the kernel of the right vertical map is the Tate module $T\Sha(\J)$ of the Shafarevich-Tate group $\Sha(\J)$ (Definition \ref{sha}).  Commutativity of this diagram immediately implies the following.

\begin{lemma}\label{tshatrivial}
Suppose that, for every $c\in C^{\cl}$, the section $s_c^{\ab}\in H^1(G_{K_c},TJ_c)$ is contained in $\widehat{J_c(K_c)}$.  Then if $T\Sha(\J)=0$, the section $s^{\ab}\in H^1(G_K,TJ)$ is contained in $\widehat{J(K)}$.
\end{lemma}

\begin{proposition}\label{sabjkhat}
Assume that $T\Sha(\J)=0$, and that $k$ strongly satisfies conditions (i), (ii) and (iii)(a) of Definition \ref{conditions}.  Then we have the following.
\begin{enumerate}
\item For each $c \in C^{\cl}$, $\bar s_c^{\ab}$ is in the image of the injective map $\X_c(k(c)) \hookrightarrow H^1(G_{k(c)},T\J_c)$, and $s_c^{\ab}$ is in the image of $X_c(K_c)\to H^1(G_{K_c},TJ_c)$.
\item $s^{\ab}$ is contained in $\widehat{J(K)}$.
\end{enumerate}
\end{proposition}

\begin{proof}
Let $\tilde s_c : G_{K_c} \to \pi_1(X_c - \tilde S_c)$ denote the section of $\pi_1(X_c - \tilde S_c)$ induced by $s_c$, and write $\varphi_{s_c} := \Sp \circ \tilde s_c$.  Let $\tilde X_{c, \tilde S_c} \to X_c - \tilde S_c$ be a universal pro-\'etale cover, and recall $\overline {k(\X_c)}$ the separable closure of the function field of $\X_c$ determined by $\xi'_c$.  By Theorem \ref{stildeptth}, for every $c\in C^{\cl}$ we have $\varphi_{s_c}(G_{K_c}) \subset D_{\tilde x}$ for a unique $x \in \X_c(k(c))$ and some unique extension $\tilde x$ of $x$ to $\overline {k(\X_c)}$, and $\tilde s_c(G_{K_c}) \subset D_{\tilde y}$ for some unique $\tilde y$ in $\tilde X_{c, \tilde S_c}$ above the unique ($K_c$-rational) point $y$ of $\tilde S_c$ specialising to $x$.  This implies that $s_c^{\et}(G_{K_c})=D_{\tilde y'}$, for some $\tilde y'$ above $y$ in a universal pro-\'etale cover $\tilde X_c \to X_c$, and $\bar s_c^{\et}(G_{k(c)}) = D_{\tilde x'}$ for some $\tilde x'$ above $x$ in a universal pro-\'etale cover $\tilde\X_c \to \X_c$.  This means that $s_c^{\et}$, respectively $\bar s_c^{\et}$ arises from $y \in X_c(K_c)$, resp. $x \in \X_c(k(c))$ by functoriality of the fundamental group, which proves (i) (see the discussion before Lemma \ref{kummerexactseq}). The map $\X_c(k(c))\to H^1(G_{k(c)},T\J_c)$ is injective by condition (iii)(a) of Definition \ref{conditions}.

Since the map $X_c(K_c)\to H^1(G_{K_c},TJ_c)$ factors through the inclusion $\widehat{J_c(K_c)} \hookrightarrow H^1(G_{K_c}, TJ_c)$ (see sequence (\ref{sabptth})), part (i) implies in particular that $s_c^{\ab}$ is contained in $\widehat{J_c(K_c)}$, and since this is true for every $c\in C^{\cl}$, Lemma \ref{tshatrivial} implies that $s^{\ab}\in\widehat{J(K)}$, which proves (ii).
\end{proof}

\section{Proof of the Main Theorems}

\subsection{Proof of Theorem A}

Let $k$ be a field of characteristic zero that satisfies conditions (iv) and (v) of Definition \ref{conditions}, and strongly satisfies conditions (i), (ii) and (iii) of Definition \ref{conditions}.  Let $C$ be a smooth, separated, connected curve over $k$ with function field $K$.  Let $\X\rightarrow C$ be a flat, proper, smooth relative curve whose generic fibre $X:=\X\times_C K$ is geometrically connected and hyperbolic, with $X(K)\ne\emptyset$.  Let $\J:=\Pic^0_{\X/C}$ denote the relative Jacobian of $\X$, and $J := \J_K$ the Jacobian of $X$.  For a closed point $c \in C^{\cl}$, denote by $\J_c := \J_{k(c)}$ the Jacobian of $\X_c$.  Assume that $T\Sha(\J)=0$.  We show that the birational section conjecture holds for $X$ (see Definitions \ref{geometricgaloissections} and \ref{bscholds} and Remark \ref{bscuniqueness}).  

Let $s:G_K\rightarrow G_X$ be a section.  Under our assumptions, the \'etale abelian section $s^{\ab}$ induced by $s$ is contained in $\widehat{J(K)}$, by Proposition \ref{sabjkhat} (ii).

\begin{lemma}\label{sabinjk}
The homomorphism $J(K)\rightarrow\widehat{J(K)}$ is injective and $s^{\ab}$ is contained in $J(K)$.
\end{lemma}

\begin{proof}
There exist $c_1,c_2\in C^{\cl}$ such that the natural specialisation homomorphism $J(K)\rightarrow\J_{c_1}(k(c_1))\times\J_{c_2}(k(c_2))$ is injective \cite[Proposition 2.4]{poonenvoloch}.  
Let $\ell $ be a finite extension of $k$ that contains $k(c_1)$ and $k(c_2)$.  Then there is an injective homomorphism $\J_{c_i}(k(c_i))\hookrightarrow\J_{c_i}(\ell )$ for each $i=1,2$, hence an injective homomorphism $\J_{c_1}(k(c_1))\times\J_{c_2}(k(c_2))\hookrightarrow\J_{c_1}(\ell)\times\J_{c_2}(\ell)\simeq(\J_{c_1}\times\J_{c_2})(\ell)$.  For ease of notation, let us write $\J_{1,2}(\ell):= (\J_{c_1}\times\J_{c_2})(\ell)$. We have a commutative diagram of exact sequences:

\[\begin{tikzpicture}[descr/.style={fill=white}]
\matrix(m)[matrix of math nodes,
row sep=3em, column sep=3em,
text height=2.5ex, text depth=0.25ex]
{0 & J(K) & \J_{1,2}(\ell) & H & 0\\
0 & \widehat{J(K)} & \widehat{\J_{1,2}(\ell)} & \widehat{H} & 0\\};
\path[->]
(m-1-1) edge (m-1-2);
\path[->]
(m-1-2) edge (m-1-3) edge (m-2-2);
\path[->,font=\scriptsize]
(m-1-3) edge (m-1-4) edge node[left]{$\phi$} (m-2-3);
\path[->]
(m-1-4) edge (m-1-5) edge (m-2-4);
\path[->]
(m-2-1) edge (m-2-2);
\path[->,font=\scriptsize]
(m-2-2) edge node[above]{$\psi$} (m-2-3);
\path[->]
(m-2-3) edge (m-2-4);
\path[->]
(m-2-4) edge (m-2-5);
\end{tikzpicture}\]
where $H$ is defined so that the upper horizontal sequence is exact.  Exactness of the lower sequence follows easily from condition (iii) (b) of Definition \ref{conditions}, while condition (iii) (a) implies that the middle and right vertical maps are injective.  Therefore the left vertical map is also injective, and the equality $J(K)=\phi(\J_{1,2}(\ell))\cap\psi(\widehat{J(K)})$ holds inside $\widehat{\J_{1,2}(\ell)}$.  

For each $c_i$, $i=1,2$, the section $s$ induces an element $\bar s_{c_i}^{\ab}\in H^1(G_{k(c_i)},T\J_{c_i})$, which is contained in the image of the map $\X_{c_i}(k(c_i)) \to H^1(G_{k(c_i)}, T\J_{c_i})$ by Proposition \ref{sabjkhat} (i).  This map is injective by condition (iii)(a) of Definition \ref{conditions}, so we may consider $\bar s_{c_i}^{\ab}$ to be contained in $\X_{c_i}(k(c_i))$.  Then $\bar s_{c_i}^{\ab}$ is contained in $\J_{c_i}(\ell)$ for each $i=1,2$, due to injectivity of the maps $\X_{c_i}(k(c_i))\hookrightarrow\J_{c_i}(k(c_i))\hookrightarrow\J_{c_i}(\ell)$.  Thus $(\bar s_{c_1}^{\ab},\bar s_{c_2}^{\ab})$ is contained in $\J_{c_1}(\ell)\times\J_{c_2}(\ell)$, hence in $\phi(\J_{1,2}(\ell))$.  By Lemma \ref{sabresinf}, the image of $s^{\ab}\in\widehat{J(K)}$ in $\widehat{\J_{1,2}(\ell)}$ under $\psi$ is the element $(\bar s_{c_1}^{\ab},\bar s_{c_2}^{\ab})$, and since this lies in $\phi(\J_{1,2}(\ell))$ we have $s^{\ab}\in\phi(\J_{1,2}(\ell))\cap\psi(\widehat{J(K)})=J(K)$.
\end{proof}

Fix a $K$-rational point $x_0\in X(K)=\X(C)$ (non-empty by assumption), and let $\iota:\X\rightarrow\J$ denote the closed immersion mapping $x_0$ to the zero section of $\J$.

\begin{lemma}\label{sabinxk}
$s^{\ab}$ is contained in $X(K)$.
\end{lemma}

\begin{proof}
Since $s^{\ab}$ is contained in $J(K) = \J(C)$, it may be regarded as a morphism $s^{\ab} : C \to \J$.  By Lemma \ref{sabresinf}, the pullback of this morphism to $\Spec k(c)$ is precisely $\bar s_c^{\ab}$, considered an element of $\J_c(k(c)) \subset H^1(G_{k(c)}, T\J_c)$.  But $\bar s_c^{\ab}$ is contained in $\X_c(k(c))$ by Proposition \ref{sabjkhat} (i), hence the morphism $\bar s_c^{\ab} : \Spec k(c) \to \J_c$ must factor through $\X_c$, where $\X_c$ is considered a closed subscheme of $\J_c$ via $\iota$.  Thus, for each $c \in C^{\cl}$, the image of $c$ under the morphism $s^{\ab} : C \to \J$ is a closed point of $\X_c$.  This implies that $s^{\ab} : C \to \J$ factors through $\iota:\X\rightarrow\J$, thus $s^{\ab}$ is contained in the subset $\iota(X(K))\subseteq J(K)$.
\end{proof}

Let $z$ be the point in $X(K)$ such that $\iota(z)=s^{\ab}$, and for each $c \in C^{\cl}$ let $\bar z_c$ denote its specialisation to $\X_c$.  Let $\bar x_c\in\X_c(k(c))$ be the point associated to $\bar s_c^{\ab}$ by Proposition \ref{sabjkhat} (i).

\begin{lemma}\label{zc=xc}
$\bar z_c=\bar x_c$ in $\X_c(k(c))$ for all $c\in C^{\cl}$.
\end{lemma}

\begin{proof}
Lemma \ref{sabresinf} implies that, for each $c \in C^{\cl}$, $\bar s_c^{\ab}$ is the image of both $\bar x_c$ and $\bar z_c$ under the map $\X_c(k(c)) \to H^1(G_{k(c)},T\J_c)$.  This map is injective by condition (iii)(a) of Definition \ref{conditions}, hence $\bar z_c=\bar x_c$.
\end{proof}

\begin{proposition}\label{sisgeometric}
$s$ is geometric.
\end{proposition}

\begin{proof}
By the ``limit argument'' of Tamagawa \cite[Proposition 2.8 (iv)]{tamagawa}, and the fact that $k$ satisfies the condition (iv) in Definition 4.1,
it suffices to prove that for any open subgroup $H \subset G_X$ which contains $s(G_K)$, if $Y \to X$ denotes the corresponding finite morphism with $Y$ smooth, we have $Y(K) \ne \emptyset$.

By construction, we have $G_Y = H$, and $s$ defines a section $s_Y : G_K \to G_Y$ of $G_Y$.  For $c\in C^{\cl}$, let $K_c$ be the completion of $K$ at $c$, and write $X_c:=X\times_K K_c$ and $Y_c:=Y\times_K K_c$.  By Lemma \ref{pullbackgx}, the section $s$ pulls back to a section $s_c:G_{K_c}\rightarrow\pi_1(X_c^{\tr})$, and likewise $s_Y$ pulls back to a section $s_{Y_c}:G_{K_c}\rightarrow\pi_1(Y_c^{\tr})$. We have the following commutative diagram.
\[\begin{tikzpicture}[descr/.style = {fill = white}, baseline = (current bounding box.center)]
\matrix(m)[matrix of math nodes,
row sep=3em, column sep=3em,
text height=2ex, text depth=0.25ex]
{\pi_1(Y_c^{\tr}) & \pi_1(X_c^{\tr}) & G_{K_c}\\
G_Y & G_X & G_K\\};
\path[right hook->]
(m-1-1) edge (m-1-2);
\path[->]
(m-1-1) edge (m-2-1);
\path[->]
(m-1-2) edge (m-1-3);
\path[->]
(m-1-2) edge (m-2-2);
\path[->,font=\scriptsize]
(m-1-3) edge[out=165,in=15] node[above left, near end]{$s_c$} (m-1-2) edge[out=145,in=25] node[below]{$s_{Y_c}$} (m-1-1);
\path[->]
(m-1-3) edge (m-2-3);
\path[right hook->]
(m-2-1) edge (m-2-2);
\path[->]
(m-2-2) edge (m-2-3);
\path[->,font=\scriptsize]
(m-2-3) edge[out=165,in=15] node[above left, near end]{$s$} (m-2-2) edge[out=150,in=25] node[above left,near end]{$s_Y$} (m-2-1);
\end{tikzpicture}\]
Let $\Y$ be the normalisation of $\X$ in the function field of $Y$, and for each $c \in C^{\cl}$ let $\Y_c$ denote the closed fibre of $\Y$ at $c$.  After possibly removing finitely many points from $C$, we may assume that $\Y$ is smooth over $C$.  Indeed, the closed fibres $\Y_c$ are smooth except possibly for finitely many closed points $c\in C^{\cl}$ \cite[Proposition 10.1.21]{Liu}.  So, if necessary, we may replace $C$ by the largest open sub-scheme  $C'\subset C$ such that $\Y_c$ is smooth for every $c\in(C')^{\cl}$, and $\Y \to \X$ by the induced map of fibre products $\Y\times_C C' \to \X\times_C C'$.

So we assume that the fibres $\Y_c$ are smooth for all $c\in C^{\cl}$.  For each closed point $\bar x_c \in \X_c^{\cl}$, respectively $\bar y_c \in \Y_c^{\cl}$, choose an algebraic point $x_c \in X_c$, resp. $y_c \in Y_c$ specialising to $\bar x_c$, resp. $\bar y_c$ whose residue field is the unique unramified extension of $K_c$ whose valuation ring has residue field $k(\bar x_c)$, resp. $k(\bar y_c)$.  Let $\tilde S_c$, respectively $\tilde T_c$ denote the set of these chosen algebraic points of $X_c$, resp. $Y_c$ (see Definition \ref{stildekc}).  The groups $\pi_1(Y_c-\tilde T_c)$ and $\pi_1(X_c-\tilde S_c)$ are naturally quotients of $\pi_1(Y_c^{\tr})$ and $\pi_1(X_c^{\tr})$ respectively, hence $s_c$ naturally induces a section $\tilde s_c:G_{K_c}\rightarrow\pi_1(X_c-\tilde S_c)$, and likewise $s_{Y_c}$ induces a section $\tilde s_{Y_c}:G_{K_c}\rightarrow\pi_1(Y_c-\tilde T_c)$.  By Theorem \ref{galspeckc}, there exist specialisation homomorphisms $\Sp_X : \pi_1(X_c - \tilde S_c) \twoheadrightarrow G_{\X_c}$ and $\Sp_Y : \pi_1(Y_c - \tilde T_c) \twoheadrightarrow G_{\Y_c}$ and a commutative diagram
\[\begin{tikzpicture}[descr/.style = {fill = white}, baseline = (current bounding box.center)]
\matrix(m)[matrix of math nodes,
column sep=3em,
text height=2ex, text depth=0.25ex]
{\pi_1(Y_c^{\tr}) & \pi_1(X_c^{\tr}) & G_{K_c}\\[2em]
\pi_1(Y_c-\tilde T_c) &  & \\[0em]
& & G_{K_c}\\[0em]
& \pi_1(X_c-\tilde S_c) & \\[3.5em]
G_{\Y_c} & G_{\X_c} & G_{k(c)}\\};
\path[right hook->]
(m-1-1) edge (m-1-2);
\path[->>]
(m-1-1) edge (m-2-1);
\path[->]
(m-1-2) edge (m-1-3);
\path[->>]
(m-1-2) edge (m-4-2);
\path[->,font=\scriptsize]
(m-1-3) edge[out=165,in=15] node[above left, near end]{$s_c$} (m-1-2) edge[out=150,in=25] node[below]{$s_{Y_c}$} (m-1-1);
\path[-]
(m-1-3) edge[double distance=2pt] (m-3-3);
\path[font=\scriptsize,->]
(m-2-1) edge[out=0, in=165] (m-3-3);
\path[font=\scriptsize,->>]
(m-2-1) edge node[left]{$\Sp_Y$} (m-5-1);
\path[->]
(m-4-2) edge[out=5, in=205] (m-3-3);
\path[font=\scriptsize,->>]
(m-4-2) edge node[left]{$\Sp_X$} (m-5-2);
\path[->,font=\scriptsize]
(m-3-3) edge[out=187,in=17] node[above left, near end]{$\tilde s_c$} (m-4-2) edge[out=150,in=10] node[above left,near end]{$\tilde s_{Y_c}$} (m-2-1) edge[out=220,in=70] node[below right]{$\varphi_X$} (m-5-2) edge[out=220,in=70] node[below right,near end]{$\varphi_Y$} (m-5-1);
\path[font=\scriptsize,->>]
(m-3-3) edge node[left]{} (m-5-3);
\path[font=\scriptsize,right hook->]
(m-5-1) edge node[above]{} (m-5-2);
\path[->]
(m-5-2) edge (m-5-3);
\end{tikzpicture}\]
where we denote $\varphi_X:=\Sp_X\circ\tilde s_c$ and $\varphi_Y:=\Sp_Y\circ\tilde s_{Y_c}$.  By Theorem \ref{stildeptth}, we have $\varphi_Y(G_{K_c}) \subset D_{\tilde y_c} \subset G_{\Y_c}$ for a unique valuation $\tilde y_c$ on $\overline {k(\X_c)}$ extending a unique $k(c)$-rational point $\bar y_c \in \Y_c(k(c))$.  By commutativity of the above diagram, this implies that $\varphi_X(G_{K_c}) \subset D_{\tilde y_c} \subset G_{\X_c}$ for the same valuation $\tilde y_c$ on $\overline {k(\X_c)}$, whose restriction to $k(\X_c)$ corresponds to the image $\bar x'_c$ of $\bar y_c$ in $\X_c$.  Thus we have found, for every $c\in C^{\cl}$, unique $k(c)$-rational points $\bar y_c\in\Y_c(k(c))$ and $\bar x'_c\in\X_c(k(c))$ such that $\bar y_c$ maps to $\bar x'_c$ via $\Y_c\rightarrow\X_c$.  Moreover, $\bar x'_c$ must be the same as the point $\bar x_c$ associated to $\bar s_c^{\ab}$ (see Lemma \ref{zc=xc} and the paragraph before it).

Recall the section $s^{\ab}$ is associated to a $K$-rational point $z$ (see Lemma \ref{sabinxk} and the paragraph after it).  View $z\in X(K) = \X(C)$ as a section $z:C\rightarrow\X$, and denote by $\Y_z$ the pullback of the image $z(C)$ via the map $\Y\rightarrow\X$.  Then $\Y_z\rightarrow z(C)$ is a finite morphism, and we can assume, after possibly shrinking $C$, that $\Y_z$ is smooth. Since $z$ specialises to $\bar x_c\in\X_c(k(c))$ (Lemma \ref{zc=xc}), $\bar x_c\in z(C)$ and therefore $\bar y_c\in\Y_z(k(c))$ for every $c\in C^{\cl}$.  Then condition (v) of Definition \ref{conditions} implies that $\Y_z(K)\ne\emptyset$.  Thus $\Y_z(K)\subseteq\Y(K)=Y(K)\ne\emptyset$, which completes the proof of Proposition \ref{sisgeometric}.
\end{proof}

Thus $s(G_k)$ is contained in a decomposition group associated to a $K$-rational point $x\in X(K)$, which is unique (cf. Remark \ref{bscuniqueness}).  This concludes the proof of Theorem A.

\subsection{Proof of Theorem B}

In this section we explain how Theorem B is deduced from Theorem A.  Let $k$ be a field of characteristic zero that strongly satisfies the conditions of Definition \ref{conditions}.  Let $C$ be a smooth, separated, connected curve over $k$ with function field $K$.  For any finite extension $L$ of $K$, let $C^L$ denote the normalisation of $C$ in $L$, and for any flat, proper, smooth relative curve $\Y \to C^L$, let $\J_{\!\Y} := \Pic^0_{\Y/C^L}$ denote the relative Jacobian of $\Y$.  Assume that for any such finite extension $L$ and any such relative curve $\Y$ we have $T\Sha(\J_{\!\Y})=0$.

We will show that for any finite extension $L$ of $K$ and any smooth, projective, geometrically connected (not necessarily hyperbolic) curve $X$ over $L$, the birational section conjecture holds for $X$.

\begin{proposition}
With the above notation and hypotheses, let $s:G_L\rightarrow G_X$ be a section of $G_X$.  Then $s$ is geometric.
\end{proposition}

\begin{proof}
By the Hurwitz formula, we may choose an open subgroup $H\subset G_X$ containing $s(G_L)$ such that, denoting by $Y\rightarrow X$ the corresponding finite morphism with $Y$ smooth, $Y$ is hyperbolic.  We have an isomorphism $H\simeq G_Y$, and $s$ naturally defines a section $s_Y:G_L\rightarrow G_Y$ of the natural projection $G_Y\twoheadrightarrow G_L$.  
Let $L'|L$ be a finite extension such that $Y(L')\ne\emptyset$, and let $M|L$ be a Galois extension of $L$ containing $L'$.  Then $Y_M(M)\ne\emptyset$, and $s_Y$ restricts to a section $s_{Y_M}:G_M\rightarrow G_{Y_M}$ of the absolute Galois group of $Y_M$.
\[\begin{tikzpicture}[descr/.style = {fill = white}, baseline = (current bounding box.center)]
\matrix(m)[matrix of math nodes,
row sep=3em, column sep=3em,
text height=1.5ex, text depth=0.25ex]
{1 & G_{Y_{\ob{L}}} & G_{Y_M} & G_M & 1\\
1 & G_{Y_{\ob{L}}} & G_Y & G_L & 1\\};
\path[->]
(m-1-1) edge (m-1-2);
\path[->,font=\scriptsize]
(m-1-2) edge (m-1-3);
\path[-]
(m-1-2) edge[double distance=2pt] (m-2-2);
\path[->]
(m-1-3) edge (m-1-4);
\path[right hook->,font=\scriptsize]
(m-1-3) edge (m-2-3);
\path[->,font=\scriptsize]
(m-1-4) edge (m-1-5) edge[out=160,in=20] node[above]{$s_{Y_M}$} (m-1-3);
\path[right hook->,font=\scriptsize]
(m-1-4) edge (m-2-4);
\path[->]
(m-2-1) edge (m-2-2);
\path[->]
(m-2-2) edge (m-2-3);
\path[->]
(m-2-3) edge (m-2-4);
\path[->,font=\scriptsize]
(m-2-4) edge (m-2-5) edge[out=160,in=20] node[above]{$s_Y$} (m-2-3);
\end{tikzpicture}\]
Let $C^M$ denote the normalisation of $C$ in $M$, and let $\Y \to C^M$ be a flat and proper model of $Y_M$ over $C^M$.  As in the proof of Proposition \ref{sisgeometric}, after possibly removing finitely many closed points from $C^M$ we may assume that the closed fibres $\Y_c:=\Y \times_{C^M} k(c)$ of $\Y$ are smooth for all $c\in (C^M)^{\cl}$.  Then $\Y\to C^M$ is a flat, proper, smooth relative curve whose generic fibre $Y_M$ is hyperbolic and has at least one $M$-rational point.  Theorem A then implies that $s_{Y_M}(G_M)$ is contained in a decomposition subgroup $D^M_{\tilde y} \subset G_{Y_M}$ for a unique $M$-rational point $y$ of $Y_M$ and some extension $\tilde y$ of $y$ to $\overline {k(X)}$.  Note we use a superscript $M$ to emphasise that $D^M_{\tilde y}$ is a subgroup of $G_{Y_M}$.

Since $M|L$ is a Galois extension, $G_M$ is a normal subgroup of $G_L$, hence $s_Y(G_L)$ normalises $s_{Y_M}(G_M)$ in $G_Y$.  Therefore, for any $\sigma \in G_L$, $s_{Y_M}(G_M)$ is also contained in $s_Y(\sigma)^{-1}D^M_{\tilde y} s_Y(\sigma) = D^M_{s_Y(\sigma)\cdot\tilde y}$, which implies that $\tilde y=s_Y(\sigma)\cdot\tilde y$ \cite[Corollary 12.1.3]{CONF}.  Thus, $s_Y(G_L)$ normalises $D^M_{\tilde y}$ in $G_Y$, so it is contained in the normaliser of $D^M_{\tilde y}$ in $G_Y$, which is precisely $D_{\tilde y} \subset G_Y$.  This implies that $s(G_L)$ is contained in the decomposition subgroup $D_{\tilde y} \subset G_X$ of the same valuation $\tilde y$ of $\overline {k(X)}$, whose restriction to $k(X)$ corresponds to the image $x$ of $y$ in $X$.  The point $x$ is then necessarily $L$-rational, since $D_{\tilde y}$ must map surjectively to $G_L$.
\end{proof}

This concludes the proof of Theorem B.

\subsection{Proof of Theorem C} 

In this section we prove Theorem C.  Assume the {\bf BSC} holds over all number fields.  We prove that the {\bf BSC} holds over all finitely generated fields over $\Q$ of transcendence degree $n \ge 1$. We argue by induction on $n$ and assume that the {\bf BSC} holds over all finitely generated fields over $\Q$ of transcendence degree $<n$. Let $K$ be a finitely generated field over $\Q$ of transcendence degree $n$ and $k\subset K$ a subfield which is algebraically closed in $K$ over which $K$ has transcendence degree $1$. We show the {\bf BSC} holds over $K$. 

It is well-known that in order to prove that the {\bf BSC} holds over $K$ it suffices to prove that the {\bf BSC} holds for the projective line over $K$ (cf. \cite[Lemma 2.1]{saidiBASC}).  Thus, we will show the following.

\begin{proposition}
With $K$ and $k$ as above, let $X=\Bbb P^1_K$, and let $s:G_K\to G_X$ a section of the projection $G_X\twoheadrightarrow G_K$.  Then $s$ is geometric.
\end{proposition}

\begin{proof}
Let $\ob{K}$ and $\ob{k}$ be the algebraic closures of $K$ and $k$, respectively, induced by the geometric point $\xi$ defining $G_X$.  We claim that there exists an open subgroup $H \subset G_X$ containing $s(G_K)$ such that, denoting by $Y \to X$ the corresponding finite morphism with $Y$ smooth, $Y$ is hyperbolic and isotrivial, meaning that $Y_{\ob{K}}$ descends to a smooth curve $Y_{\bar k}$ over $\bar k$.
Indeed, let $U\subset \Bbb P^1_k$ be an open subset, $U_{\bar k}=U\times _k \bar k$, $U_K=U\times _k K$, and $U_{\ob{K}}=U\times _k \ob{K}$.  We have a natural commutative diagram of exact sequences
$$
\CD
1@>>> \pi_1(U_{\ob{K}},\bar \xi) @>>> \pi_1(U_K,\xi) @>>> G_{K} @>>> 1\\
@. @VVV  @VVV   @VVV @. \\
1@>>> \pi_1(U_{\bar k},\bar \xi) @>>> \pi_1(U_k,\xi)@>>> G_{k}@>>> 1\\
\endCD
$$
where the left vertical map is an isomorphism. Let $\tilde \Delta$ be a characteristic open subgroup of $\pi_1(U_{\bar k},\bar \xi)$ corresponding to a finite morphism $Y_{\bar k}\to \Bbb P^1_{\bar k}$ with $Y_{\bar k}$ smooth and hyperbolic.  Such a subgroup exists by the Riemann-Hurwitz formula and the fact that $\pi_1(U_{\bar k},\bar \xi)$ is finitely generated. We write $\Delta$ for the corresponding subgroup of $\pi_1(U_{\ob{K}},\bar \xi)$ and $Y_{\ob{K}}\to \Bbb P^1_{\ob{K}}$ the corresponding finite morphism with $Y_{\ob{K}}$ smooth. The section $s$ induces a section $s_U : G_K \to \pi_1(U_{K},\xi)$ of the projection $\pi_1(U_K,\xi) \twoheadrightarrow G_{K}$. Let $\widetilde H=\Delta \cdot s_U(G_K)$ and $H$ the inverse image of $\widetilde H$ in $G_X$. Then $H$ and the corresponding finite morphism $Y\to X$ are as claimed above.


The section $s$ induces a section $s_Y : G_K \to G_Y = H$ of the natural projection $G_Y\twoheadrightarrow G_K$ with $s_Y(G_K) = s(G_K)$, and one easily verifies that the section $s$ is geometric if (and only if) the section $s_Y$ is geometric.  Let $C$ be a separated, smooth and connected curve over $k$ with function field $k(C)=K$, and $\Y\to C$ a flat, smooth and proper relative $C$-curve with generic fibre $\Y_K=Y$. Without loss of generality, we can assume that $Y(K)\ne \emptyset$ (cf. proof of Theorem B). Let $\J$ be the relative jacobian of $\Y$. The Shafarevich-Tate group $\Sha (\J)$ is finite by \cite[Theorem 4.1]{saiditamagawa} since $\J_{\ob{K}}$, being the jacobian of $Y_{\ob{K}}$, is isotrivial (i.e. descends to an abelian variety over $\bar k$).  Moreover, $k$ strongly satisfies the conditions in Definition \ref{conditions}, since it is finitely generated and by the above induction assumption. Then the section $s_Y$, and a fortiori the section $s$, is geometric by Theorem A.
\end{proof}

This concludes the proof of Theorem C.

\addcontentsline{toc}{section}{References}

\paragraph{}
\begin{flushleft}
Mohamed Sa\"idi

m.saidi@exeter.ac.uk
\end{flushleft}

\begin{flushleft}
Michael Tyler

m20tyler@gmail.com
\end{flushleft}

\end{document}